\documentclass{amsart}

\usepackage{amssymb,color}
\usepackage{amsfonts}
\usepackage{amsmath}
\usepackage{euscript}
\usepackage{enumerate}
\usepackage{mathabx}
\usepackage{pdfsync}
\newtheorem{theorem}{Theorem}[section]
\newtheorem{lemma}[theorem]{Lemma}

\newtheorem{prop}[theorem]{Proposition}

\newtheorem*{Theorem1'}{Theorem 1'}

\theoremstyle{definition}

\theoremstyle{remark}

\renewcommand \H{{\mathrm{Hol}}}
\newcommand \Aut{{\mathrm{Aut}}}
\newcommand \Z{{\mathbb Z}}

\newcommand \N{{\mathbb N}}
\newcommand \GL{{\mathrm{GL}}}

\newcommand \Heis{{\mathrm{Heis}}}

\def\A{{\mathfrak{a}}}
\def\B{{\mathfrak{b}}}
\def\C{{\mathfrak{c}}}
\def\D{{\mathfrak{d}}}
\def\S{{\Sigma}}
\def\T{{\Psi}}
\def\U{{\Upsilon}}
\def\So{{\Sigma}_0}

\begin{document}

\title[the automorphism group of certain 3-generator $p$-groups]{the automorphism group of certain 3-generator $p$-groups}

\author{Fernando Szechtman}
\address{Department of Mathematics and Statistics, University of Regina, Canada}
\email{fernando.szechtman@gmail.com}
\thanks{This research was partially supported by NSERC grant RGPIN-2020-04062}





\begin{abstract} A presentation as well as a structural description of the automorphism group
of a family of 3-generator finite $p$-groups is given, $p$ being an odd prime.
\end{abstract}

\maketitle

\section{Introduction}

We fix a prime number $p$ throughout the paper. The automorphism groups of finite $p$-groups have been widely studied. One line of investigation was
related to the divisibility conjecture, to the effect that if $G$ is a finite noncyclic $p$-group of order larger than $p^2$, then
$|G|$ divides $|\mathrm{Aut}(G)|$. While this
holds true for several classes of $p$-groups, it has recently been shown to be false in general \cite{GJ}. A thorough survey of this conjecture can be found in \cite{Di}.
A second avenue of research is connected to the conjecture that every finite nonabelian $p$-group has a noninner automorphism of order
$p$. We refer the reader to \cite{G}
and references therein for the status of this conjecture. A third focus of attention has been the actual structure of $\mathrm{Aut}(G)$ when $G$ is a finite $p$-group of one type or another. This has been the case, in particular for metacyclic groups, as found in
\cite{BC,C,C2,GG,M,M2,ZL}, for instance.

The present paper gives a presentation as well as a structural description of the
automorphism group
of a family of 3-generator finite $p$-groups, with $p$ odd.

We assume henceforth that $p$ is odd and that $i,n\in\N$ satisfy $n\geq 2$ and $1\leq i\leq n-1$. Given an integer $d$
that is not a multiple of $p$, the inverse of $1+dp$ modulo $p^n$ is also of the form $1+ep$ for some
integer $e$ not divisible by $p$, so we have
\begin{equation}\label{eqrd}
p\nmid d,\; p\nmid e,
\end{equation}
\begin{equation}\label{eqrd2}
(1+dp)(1+ep)\equiv 1\mod p^n.
\end{equation}
Let $\S$ be the group with the following presentation, depending on whether $2i\leq n$ or~$2i\geq n$:
\begin{equation}
\label{sime}
\S=\langle \A,\B,\C\,|\, \A^{p^i}=\B^{p^{n-i}}=\C^{p^{n-1}}=1, [\A,\B]=\C^{p^{n-i-1}}, \C\A \C^{-1}=\A^{1+e p},
\C \B\C^{-1}=\B^{1+d p}\rangle,
\end{equation}
\begin{equation}
\label{sima}
\S=\langle \A,\B,\C\,|\, \A^{p^{n-i}}=\B^{p^{n-i}}=\C^{p^{n-1}}=1, [\A,\B]=\C^{p^{i-1}}, \C\A \C^{-1}=\A^{1+e p},
\C \B\C^{-1}=\B^{1+d p}\rangle.
\end{equation}
This is a quotient of the finite Wamsley group \cite{W} defined on 3 generators with 3 relations.

In this paper, we give a presentation of $\mathrm{Aut}(\S)$ and its normal Sylow $p$-subgroup $\Pi$,
and provide a structural description of these groups.

There is a noteworthy connection between $\S$ and automorphism groups. Let $\U$
be the semidirect product of a cyclic group $C_{p^n}$
of order $p^n$ by the unique subgroup of order $p^{n-i}$ of the cyclic group $\mathrm{Aut}(C_{p^n})$. Thus
\begin{equation}
\label{preshi}
\U=\langle x,y\,|\, x^{p^n}=1, y^{p^{n-i}}=1, yxy^{-1}=x^{1+p^i}\rangle.
\end{equation}
Then $\S$ is isomorphic to the normal Sylow $p$-subgroup of $\mathrm{Aut}(\U)$.
A presentation for $\mathrm{Aut}(G)$ when $G$ is a split metacyclic $p$-group can be found in \cite{BC}. This includes, in particular, the case when $G=\U$.

Consider the group $\T$ with the following presentation, depending on whether $2i\leq n$ or $2i\geq n$:
$$
\T=\langle
\A,\B,\D\,|\,
\A^{p^i}=\B^{p^{n-i}}=\D^{p^i}=1, [\A,\B]=\D, \A\D=\D\A,
\B\D=\D\B
\rangle,
$$
$$
\T=\langle
\A,\B,\D\,|\,
\A^{p^{n-i}}=\B^{p^{n-i}}=\D^{p^{n-i}}=1, [\A,\B]=\D, \A\D=\D\A,
\B\D=\D\B
\rangle.
$$
Then $\T$ is a $p$-group of Heisenberg type, and $\S$ is an extension of $\T$ by a cyclic factor.
Moreover, $\S$ cannot be generated by fewer than 3 elements, except in the extreme case $n=2$, when $\S$ is the Heisenberg group of order $p^3$. Furthermore, $\S$ is a nonsplit extension of $\T$, except when $n=2$, in which case $\S=\T$.

The fact that $\S$ requires 3 generators makes
the computation of $\mathrm{Aut}(\S)$ particularly challenging. An appreciation of the obstacles involved can perhaps be gleaned from the following:
$\Pi$ requires 6 generators and 21 relations when $2i\geq n$ but $i\neq n-1$, while if $2i<n$ we utilize
8 generators and 38 relations for this purpose (although we could manage with slightly fewer). We feel that the groups $\S$ constitute a nontrivial
addition to the reservoir of $p$-groups for which the automorphism group is known.

Our work begins in Section \ref{bast}, where we describe properties of $\S$ and $\T$.
Once this background information has been collected, Sections \ref{rojo1} and \ref{rojo2} describe $\mathrm{Aut}(\S)$ and $\Pi$. In Section \ref{rojo1} we assume
that $2i\geq n$; in this case, $\T$ is a characteristic subgroup of $\S$ and
our analysis of the image of the restriction homomorphism
$\mathrm{Aut}(\S)\to \mathrm{Aut}(\T)$ allows us in Theorem~\ref{mainresult} to obtain the desired presentations of
$\mathrm{Aut}(\S)$ and $\Pi$ when $i<n-1$. Considerable effort is spent in Theorem \ref{mainresult} in finding suitable generators for $\Pi$,
and we show that no smaller amount of generators is possible. In addition, Theorem \ref{mainresult} gives a presentation of 
$\mathrm{Out}(\S)$, and shows that $\Pi$ is the kernel of the canonical homomorphism  $\mathrm{Aut}(\S)\to \mathrm{Aut}(\S/\S^p)$. The case $i=n-1$
is essentially different and is handled in Theorem \ref{lindop}, which gives presentations of Bruhat type for $\mathrm{Aut}(\S)$
and $\mathrm{Aut}(\T)$. 
In Section \ref{rojo2} we suppose
that $2i<n$; this case is considerably harder, as $\T$ is not a characteristic subgroup of $\S$. Let $G$ be the subgroup of
$\mathrm{Aut}(\S)$ preserving~$\T$. Theorem~\ref{main2} shows that $\mathrm{Aut}(\S)=G\langle U\rangle$, where $U$ is a suitable automorphism of $\S$,
and finds generators for $G$ and hence for $\mathrm{Aut}(\S)$. Presentations and further descriptions of $\mathrm{Aut}(\S)$ and
$\Pi$ can be found in Theorem \ref{main3} (which omits some relations from the case $n=2i+1$).
It turns out that $\Pi=\Phi\langle U\rangle$, where $\Phi$ is the kernel of the canonical homomorphism
$\mathrm{Aut}(\S)\to \mathrm{Aut}(\S/\S^p)$. Moreover, $\Pi$ requires at least 7 generators. 

An appendix contains a study of $\mathrm{Aut}(\U)$, independent of \cite{BC},
and establishes the foregoing connection between $\S$ and $\mathrm{Aut}(\U)$. In particular, we somewhat simplify the presentation
of $\mathrm{Aut}(\U)$ derived in \cite{BC} for an arbitrary
split metacyclic $p$-group.


Given a finite $p$-group $P$, we write $z(P)$ for the minimum number of generators of $P$, that is, $z(P)$
is the dimension of the $\Z/p\Z$-vector space $P/\Phi(P)$, where $\Phi(P)$ is the  Frattini subgroup of~$P$.
If $G$ is an arbitrary finite group, we let $s(G)$ denote a Sylow $p$-subgroup
of $G$. Thus $z(\U)=2$, $s(\mathrm{Aut}(\U))=\S$ and $z(\S)=3$ when $n>2$; also
$s(\mathrm{Aut}(\S))=\Pi$ with $z(\Pi)=6$ if $2i\geq n$, $i\neq n-1$, and $z(\Pi)=7$ if $2i<n$.
On the other hand, if $P$ is an elementary abelian group of order $p^m$, then $z(P)=m$ and $z(s(\mathrm{Aut}(P)))=m-1$.
Prompted by these examples and others, and inspired by the divisibility conjecture, we wonder if it might be of interest to determine what classes of finite $p$-groups $G$ satisfy $z(s(\mathrm{Aut}(G)))/z(G)\geq 1$, including further information about this ratio.

Given a group $G$, we write
$$
[a,b]=aba^{-1}b^{-1},\quad a,b\in G,
$$
and let $\delta:G\to \mathrm{Inn}(G)$ stand for the homomorphism
$a\mapsto \delta_a$,
where $$\delta_a(b)={}^{a} b=aba^{-1},\quad a,b\in G.$$
The image of $\gamma\in \mathrm{Aut}(G)$ under the natural projection $\mathrm{Aut}(G)\to \mathrm{Out}(G)$
will be denoted by $\overline{\gamma}$.
Observe that
$$
[ab,c]={}^{a}[b,c][a,c],\; [a,b]^{-1}=[b,a],\; [c,ab]=[c,a]\, {}^{a}[c,b].
$$
In particular, if $[G,G]$ is included in the center of $G$, then the bracket
is a group homomorphism in each variable. We will repeatedly and implicitly use these facts throughout the paper.

\section{Background material on $\S$ and $\T$}\label{bast}

\begin{prop}\label{u}  The group $\S$ cannot be generated by fewer than 3 elements, unless $n=2$. 
\end{prop}

\begin{proof} Let $E$ be an elementary abelian $p$-group of order $p^3$ and let $a,b,c$ be generators of $E$.
Consider the assignment
\begin{equation}
\label{ass}
\A\mapsto a,\; \B\mapsto b,\; \C\mapsto c.
\end{equation}
The given presentation of $\S$ ensures that, except when $n=2$, we can extend (\ref{ass}) to an
epimorphism $\S\to E$. Since $E$ cannot be generated by fewer than 3 elements, neither can $\S$.
\end{proof}

\begin{prop}\label{v} The derived subgroup of $\S$ is generated by $\A^p,\B^p,\C^{p^{n-i-1}}$ if $2i\leq n$, and
by $\A^p,\C^p,\C^{p^{i-1}}$ if $2i\geq n$.
\end{prop}

\begin{proof} Let $T$ be the subgroup of $\S$ generated by the stated elements. The given presentation of $\S$ and (\ref{eqrd})
ensure that $T\subseteq [\S,\S]$, $T$ is normal in $\S$, and
$\S/T$ is abelian, whence $T=[\S,\S]$. 
\end{proof}

Given a ring $R$ with $1\neq 0$ and a right module $M$, the Heisenberg group $\Heis(R,M)$
consists of all matrices
$$
\left(\begin{array}{ccc}
1 & u & v\\
0 & 1 & r\\
0 & 0 & 1
\end{array}\right),\quad r\in R, u,v\in M,
$$
under the usual matrix multiplication, with the understanding with $1\cdot u=u$ for $u\in M$. If $M=R$ we simply write
$\Heis(R)=\Heis(R,M)$.

\begin{prop}\label{heipres} We have $\T\cong\Heis(\Z/p^{n-i}\Z,\Z/p^i\Z)$
if $2i\leq n$ and $\T\cong\Heis(\Z/p^{n-i}\Z)$ if $2i\geq n$. In particular, $\T$ has order $p^{n+i}$ if $2i\leq n$
and $p^{3(n-i)}$ if $2i\geq n$, and any element of $\T$ can be written in one and only one way in the form $ABD$, where
$A\in\langle\A\rangle$, $B\in\langle\B\rangle$, and $D\in\langle\D\rangle$
\end{prop}

\begin{proof} Note that
the annihilator of the $\Z$-module $\Z/p^i\Z$ is $p^i\Z$; if $2i\leq n$ then $p^{n-i}\Z$ is contained in~$p^i\Z$, so $\Z/p^i\Z$ becomes a $\Z/p^{n-i}\Z$-module. The given presentation of $\T$ yields the decomposition $\T=\langle\A\rangle\langle\B\rangle\langle\D\rangle$,
and ensures that the assignment
$$
\A\mapsto \left(\begin{array}{ccc}
1 & 1 & 0\\
0 & 1 & 0\\
0 & 0 & 1
\end{array}\right),\; \B\mapsto \left(\begin{array}{ccc}
1 & 0 & 0\\
0 & 1 & 1\\
0 & 0 & 1
\end{array}\right),\; \D\mapsto \left(\begin{array}{ccc}
1 & 0 & 1\\
0 & 1 & 0\\
0 & 0 & 1
\end{array}\right)
$$
extends to a group epimorphism $\T\to \Heis(\Z/p^{n-i}\Z,\Z/p^i\Z)$ if $2i\leq n$ and
$\T\to\Heis(\Z/p^{n-i}\Z)$ if $2i\geq n$. It is an isomorphism as the presentation of $\T$
bounds its order, from above, by the order of the corresponding Heisenberg group. Uniqueness
follows from existence and $|\T|$.
\end{proof}

\begin{lemma}
\label{uno} Let $a,b,c,\ell\in\Z$, where $a,c\geq 1$ and $b\geq 0$. Then
\begin{equation}
\label{abc}
(1+\ell p^a)^{c p^b}\equiv 1+c \ell p^{a+b}\mod p^{2a+b}.
\end{equation}
\end{lemma}

\begin{proof} We assume first that $c=1$ and show (\ref{abc}) by induction on $b$. If $b=0$ there is nothing to do.
Suppose (\ref{abc}) is true for some $b\geq 0$. Then there exists $s\in\Z$ such that
$$
(1+\ell p^a)^{p^b}=1+\ell p^{a+b}+s p^{2a+b}=1+p^{a+b}(\ell+sp^a).
$$
Set $f=\ell+sp^a$. Then
$$
(1+\ell p^a)^{p^{b+1}}=(1+p^{a+b}f)^p=1+p^{a+b+1}f+{{p}\choose{2}} p^{2(a+b)}f^2+\cdots+{{p}\choose{p}} p^{p(a+b)}f^p.
$$
Since $p$ is odd, we have $p|{{p}\choose{2}}$, so there is some $k\in\Z$ such that
$$
{{p}\choose{2}} p^{2(a+b)}f^2=p^{2a+2b+1}k.
$$
Moreover, since $a\geq 1$, we have $i(a+b)\geq 2a+b+1$ for all $3\leq i\leq p$. Therefore,
$$
(1+\ell p^a)^{p^{b+1}}\equiv 1+p^{a+b+1}(\ell+sp^a)\equiv 1+\ell p^{a+(b+1)}\mod p^{2a+(b+1)}.
$$
This proves (\ref{abc}) when  $c=1$. Now if $c\in\N$ is arbitrary then the previous case yields
$$
(1+\ell p^a)^{cp^b}=(1+p^{a+b}(\ell+hp^a))^c
$$
for some $h\in\Z$. Since $i(a+b)\geq 2a+b$ for all $i\geq 2$, the binomial expansion implies
$$
(1+p^{a+b}(\ell+hp^a))^c\equiv 1+cp^{a+b}(\ell+hp^a)\equiv 1+c\ell p^{a+b}\mod p^{2a+b}.
$$
This demonstrates (\ref{abc}) in general.
\end{proof}

\begin{prop}\label{heipres2} The group $\S$ is an extension of $\T$ by $C_{p^{n-i-1}}$
if $2i\leq n$ and by $C_{p^{i-1}}$ if $2i\geq n$. In particular, $\S$ has order $p^{2n-1}$ if $2i\leq n$
and $p^{3n-2i-1}$ if $2i\geq n$, and any element of $\S$ can be written in one and only one way in the form $ABC$, where
$A\in\langle\A\rangle$, $B\in\langle\B\rangle$, and $C\in\langle\C\rangle$.
\end{prop}

\begin{proof} By (\ref{eqrd}) and (\ref{eqrd2}) there is an automorphism $\sigma$ of $\T$ such that
$\A\mapsto \A^{1+ep}$, $\B\mapsto \B^{1+dp}$ and $\D\mapsto \D$. Setting $j=n-i-1$ if $2i\leq n$ and $j=i-1$
if $2i\geq n$, Lemma \ref{uno} implies that
$\sigma^{p^j}=1_\T=\delta_\D$. As $\sigma(\D)=\D$, there is an
extension $G=\T\langle \C\rangle$ of $\T$ with cyclic factor $C_{p^j}$ and such that conjugation by $\C$ acts on $\T$
via $\sigma$ (see \cite[Chapter III, Section 7]{Z}). The given presentation of $\S$ yields an epimorphism
$\S\to G$, which is an isomorphism as $|\S|\leq |G|$. That $\S=\langle\A\rangle\langle\B\rangle\langle\C\rangle$
follows from $G=\T\langle \C\rangle$ and Proposition \ref{heipres}, while uniqueness follows from existence and the given order of $\S$.
\end{proof}

We henceforth view $\T$ as a normal subgroup of $\S$, as indicated in (the proof of) Proposition~\ref{heipres2}.

\begin{lemma} \label{potpq} Let $A\in\langle\A\rangle$, $B\in\langle\B\rangle$, $D\in\langle\D\rangle$,
and $C\in\langle\C\rangle$. Then
$$
(ABDC)^p=A_0B_0D_0C^p,
$$
for some $A_0\in\langle\A^p\rangle$, $B_0\in\langle\B^p\rangle$, and $D_0\in\langle\D^p\rangle$.
\end{lemma}

\begin{proof} We have
$$
(ABDC)^p=w C^p,
$$
where
$$
w=AB\cdot {}^{C}\!(AB)\cdot {}^{C^2}\!(AB)\cdots  {}^{C^{p-1}}\!(AB)D^p.
$$
Now $C=\C^r$, with $r\in\N$, so setting
$$
j=(1+ep)^r,\; k=(1+dp)^r,
$$
where $d$ and $e$ are as in (\ref{eqrd}) and (\ref{eqrd2}), the defining relations of $\S$ then give
$$
w=ABA^jB^kA^{j^2}B^{k^2}\cdots A^{j^{p-1}}B^{k^{p-1}}D^p.
$$
Recalling the comments on commutators made in the Introduction, we find that
$$
w=A^{1+j+\cdots+j^{p-1}}B^{1+k+\cdots+k^{p-1}}D^p D_1,
$$
where
$$
D_1=[B,A]^j [B,A]^{(1+k)j^2}\cdots [B,A]^{(1+k+\cdots+k^{p-2})j^{p-1}}.
$$
Since
$$
j\equiv 1\equiv k\mod p,
$$
we have
$$
1+j+\cdots+j^{p-1}\equiv 0\equiv 1+k+\cdots+k^{p-1}\mod p
$$
and
$$
j+(1+k)j^2+\cdots+(1+k+\cdots+k^{(p-2)})j^{p-1}\equiv 1+2+\cdots+(p-1)\equiv 0\mod p.
$$
Thus $w=A_0B_0D_0$, with $A_0\in\langle\A^p\rangle$, $B_0\in\langle\B^p\rangle$, and $D_0\in\langle\D^p\rangle$.
\end{proof}

If $G$ is a group, then $G^p$ stands for the subgroup of $G$ generated by all elements $g^p$, with $g\in G$.

\begin{prop}\label{potp} Suppose $n\neq 2$ and let $v\in\S^p$. Then
$v=ABC$ for unique elements
$A\in\langle\A^p\rangle$, $B\in\langle\B^p\rangle$, and $C\in\langle\C^p\rangle$.
\end{prop}

\begin{proof} As $n\neq 2$, we have $\D\in \langle \C^p\rangle$. Let $u\in\S$. Then
Proposition \ref{heipres2} and Lemma \ref{potpq} give
\begin{equation}
\label{cuza}
u^p=A_0B_0C_0,
\end{equation}
where
$A_0\in\langle\A^p\rangle$, $B_0\in\langle\B^p\rangle$, and $C_0\in\langle\C^p\rangle$.
As $\S$ has finite order, $v$ is a finite product of elements of the form (\ref{cuza}), and it
is easy to see that such a product will be also be of the stated form. This proves existence,
while uniqueness follows from Proposition \ref{heipres2}.
\end{proof}

Given a nonzero integer $m$, we write $v_p(m)$ for the $p$-valuation of $m$,
so that $v_p(m)=a$,
where $a$ is the unique integer satisfying $a\geq 0$, $p^a|m$, and $p^{a+1}\nmid m$. We extend the use of $v_p$
to nonzero rational numbers in the usual way.

\begin{prop}\label{charac} Let $\Omega\in\mathrm{Aut}(\S)$ and $t\in \T$. Set $\ell=i$ if $2i\leq n$ and $\ell=n-i$ if $2i\geq n$.
Suppose that $t^{p^\ell}=1$. Then $\Omega(t)\in \T$. In particular, if $2i\geq n$ then $\T$ is a characteristic subgroup of $\S$.
\end{prop}

\begin{proof} If $n=2$ then $\T=\S$, so we assume that $n\neq 2$.

Let $\Omega\in\mathrm{Aut}(\S)$ and $t\in \T$ be arbitrary.  Then
$t^p\in\langle\A^p,\B^p,\D^p\rangle$ by Lemma \ref{potpq}. By Proposition~\ref{v}, $\langle\A^p,\B^p,\D\rangle=[\S,\S]$,
which is a characteristic subgroup of $\S$. Thus, setting $u=\Omega(t)$, we have
$$
u^p=\Omega(t)^p=\Omega(t^p)\in \langle\A^p,\B^{p},\D\rangle\subset \T.
$$
On the other hand, by Proposition \ref{heipres2}, we have
\begin{equation}
\label{luca}
u=AB\C^r,\quad r\in\Z,
\end{equation}
where $A\in\langle \A\rangle$ and $B\in\langle \B\rangle$. Thus Lemma \ref{potpq}  gives
$$
u^p=A_0B_0D_0\C^{rp},
$$
where $A_0\in\langle \A^p\rangle$, $B_0\in\langle \A^p\rangle$, and $D_0\in\langle \D^p\rangle$.
From $u^p\in \T$ we infer $\C^{rp}\in \T$. Since $\C^{p^{n-\ell-1}}=\D\in \T$ but $\C^{p^{n-\ell-2}}\notin \T$,
we deduce $v_p(r)\geq n-\ell-2$, so
\begin{equation}
\label{luca2}
r=p^{n-\ell-2}f,\quad f\in\Z,
\end{equation}
and therefore
$$
u^p=A_0B_0D_0\D^f.
$$
We now suppose for the first time that $t^{p^\ell}=1$, so that $u^{p^\ell}=1$. Then
$$
(A_0B_0D_0\D^f)^{p^{\ell-1}}=1,
$$
that is
$$
(A_0B_0D_0)^{p^{\ell-1}}\D^{f p^{\ell-1}}=1.
$$
Referring to the normal form of the elements of $\T$, the $\D$-component of $(A_0B_0D_0)^{p^{\ell-1}}$ is equal to
$$
D_0^{p^{\ell-1}}[B_0,A_0] [B_0,A_0]^2\cdots [B_0,A_0]^{p^{\ell-1}-1}=1.
$$
Therefore $\D^{f p^{\ell-1}}=1$, which implies $p|f$. Thus $v_p(r)\geq n-\ell-1$ and $u\in \T$.
\end{proof}

\begin{prop}\label{nonsplit} The group $\S$ is a nonsplit extension of $\T$, unless $n=2$ in which case $\T=\S$.
\end{prop}

\begin{proof} We may assume $n\neq 2$. Suppose, if possible, that there is an element $u\in \S$ such that
$\S=\T\rtimes \langle u\rangle$, so that $u$ has order $q$, with $q=p^{n-i-1}$ if $2i\leq n$ and $q=p^{i-1}$ if $2i\geq n$.
Now $u=t\C^j$, where $t\in \T$ and $j\in\Z$.
Here $p\nmid j$, for otherwise  $\C\notin \T\rtimes\langle u\rangle$. Proposition \ref{heipres} and a repeated application of Lemma \ref{potpq} yield
$$
1=u^q=A_1B_1D_1\C^{jq}=A_1B_1D_1\D^{j},
$$
where $A_1\in\langle\A^p\rangle$, $B_1\in\langle\B^p\rangle$, and $D_1\in\langle\D^p\rangle$. The normal form of the elements of $\T$ forces $D_1\D^j=1$. As $p\nmid j$, we see that $D_1\D^j$ generates the nontrivial group $\langle \D\rangle$, a contradiction.
\end{proof}

\begin{prop} Let $(d_0,e_0)$ be a pair of integers satisfying (\ref{eqrd}) and
(\ref{eqrd2}), and let $\S_0$ be the group associated with it via (\ref{sime}) and (\ref{sima}). Then
$\S\cong\S_0$.
\end{prop}

\begin{proof} Set $j=i$ if $2i\leq n$ and $j=n-i$ if $2i\geq n$.

Let $H$ be the Sylow $p$-subgroup of the unit group $(\Z/p^{n-1}\Z)^\times$. Then $H$ is generated by the class
of $1+ep$. Since the class of $1+e_0p$ also generates $H$, there is an integer $r$, relatively prime to $p$, such that
$(1+ep)^r\equiv 1+e_0p\mod p^{n-1}$. Taking inverses modulo $p^{n-1}$, we infer $(1+dp)^r\equiv 1+d_0p\mod p^{n-1}$.
Consider the elements $A=\A^r,B=\B, C=\C^r$. They generate $\S$ and satisfy
$$
A^{p^j}=B^{p^{n-i}}=C^{p^{n-1}}=1,\; [A,B]=C^{n-j-1},\; CAC^{-1}=A^{1+e_0p},\; CBC^{-1}=B^{1+d_0p}.
$$
This readily gives an isomorphism $\S_0\to\S$.
\end{proof}

For future reference note that (\ref{eqrd2}) implies $e\equiv -d\mod p$. Since $p$ is odd and $p\nmid d$, we deduce 
\begin{equation}\label{enod}
e\not\equiv d\mod p.
\end{equation}


\section{The automorphism group of $\S$ when $2i\geq n$}\label{rojo1}

We assume throughout this section that $2i\geq n$. Using the matrix description of $\T$ given in Proposition \ref{heipres}, we readily see that
$$
Z(\T)=\langle \D\rangle.
$$
In particular, $\D=\C^{p^{i-1}}$ is central in $\S$. More precisely, we have the following result.

\begin{lemma}\label{centro} The center of $\S$ is generated by $\C^{p^{n-i-1}}$.
\end{lemma}

\begin{proof} Let $z\in Z(\S)$. Then $z=\C^r t$ for some $t\in \T$ and $r\in\N$, so
(\ref{sima}) yields
$$
{}^{t}\A=\A\D^a,\; {}^{t}\B=\B\D^b
$$
for some $a,b\in\Z$, as well as
$$
\A={}^{z}\A=\,{}^{\C^r}\!(\A\D^a)=\A^{(1+ep)^r}\D^a,\;
\B={}^{z}\B={}^{\C^r}\!(\B\D^b)=\B^{(1+dp)^r}\D^b,
$$
so $t\in Z(\T)=\langle \D\rangle$ and $\C^r$ commutes with $\A$ and $\B$.
Since $\C^r$ commutes with $\B$,
we see  that $(1+d p)^r\equiv 1\mod p^{n-i}$.  As $p\nmid d$, Lemma \ref{uno} ensures that $1+dp$ has order $p^{n-i-1}$
modulo $p^{n-i}$, so $p^{n-i-1}|r$. But $i-1\geq n-i-1$, so $\D\in \langle \C^{p^{n-i-1}}\rangle$, and by above
$Z(\S)\subseteq \langle \C^{p^{n-i-1}}\rangle$. As the reverse inclusion is clear, the result follows.
\end{proof}

\begin{theorem}\label{mainresult} Let $p$ be an odd prime. Suppose $i,n\in\N$ satisfy
$n\geq 2$, $1\leq i\leq n-1$, and $2i\geq n$. Let $d,e\in\Z$ be chosen so that
(\ref{eqrd}) and (\ref{eqrd2}) hold, and let $\S$ be the group with presentation (\ref{sima}).

Let $A,B,C\in\mathrm{Inn}(\S)$ be the inner automorphisms of $\S$ associated with $\A,\B,\C$, respectively.
Let $D,E,F,G,H\in \mathrm{Aut}(\S)$ be respectively the defined by
$$
\A \mapsto \A^{1+p^{n-i-1}},\;\B\mapsto\B,\; \C\mapsto \C^{1+p^{n-i-1}},
$$
$$
\A\mapsto\A\B^{p^{n-i-1}},\; \B\mapsto\B,\; \C\mapsto\C,
$$
$$
\A\mapsto\A,\; \B\mapsto\A^{p^{n-i-1}}\B,\; \C\mapsto\C,
$$
$$
\A\mapsto\A^{g_0},\; \B\mapsto\B^{h_0},\; \C\mapsto\C,
$$
$$
\A\mapsto\B,\; \B\mapsto\A,\; \C\mapsto\C^{-1},
$$
where $g_0$ is any integer of order $p-1$ modulo $p^{n-i}$, with inverse $h_0$ modulo $p^{n-i}$. Consider the homomorphism
$\lambda:\mathrm{Aut}(\S)\to \mathrm{Aut}(\S/\T)$, whose existence is ensured by Proposition \ref{charac}, as well as the normal subgroup
$\Delta=\ker\lambda$ of $\mathrm{Aut}(\S)$. Moreover, suppose that $i\neq n-1$. Then
\begin{equation}
\label{autoabc}
\mathrm{Aut}(\S)=\langle A,B,C,D,E,F,G,H\rangle
\end{equation}
is a group of order $2p^{3n-2i+1}(p-1)$, and
$$
\Delta=\langle A,B,C,D,E,F,G\rangle,\;
\mathrm{Aut}(\S)=\Delta\rtimes\langle H\rangle=\langle A,B,C,D,E,F\rangle\rtimes
(\langle G\rangle\rtimes\langle H\rangle),\; [\mathrm{Aut}(\S):\Delta]=2,
$$
$$
\Delta=\mathrm{Inn}(\S)\rtimes \langle D,E,F,G\rangle,\;
$$
$$
\mathrm{Inn}(\S)=\langle A,B,C\rangle\cong (\Z/p^{n-i}\Z\times \Z/p^{n-i}\Z)\rtimes \Z/p^{n-i-1}\Z,
$$
$$
\langle D,E,F,G\rangle\cong \Z/p^{i}\Z\times [(\Z/p\Z\times \Z/p\Z)\rtimes \Z/(p-1)\Z].
$$
Furthermore, the generators $A,B,C,D,E,F,G,H$ of $\mathrm{Aut}(\S)$ satisfy the defining relations
$$
A^{p^{n-i}}=1,\; B^{p^{n-i}}=1,\; C^{p^{n-i-1}}=1,\; CAC^{-1}=A^{1+ep},\; CBC^{-1}=B^{1+dp},\; AB=BA,
$$
$$
D^{p^{i}}=1,\; E^{p}=1,\; F^{p}=1,\; G^{p-1}=1,\; H^2=1,\; DE=ED,\; DF=FD,\; DG=GD,\; EF=FE,
$$
$$
GEG^{-1}=E^{h_0^2},\; GFG^{-1}=F^{g_0^2}, HEH^{-1}=A^{-p^{n-i-1}}F,\; HFH^{-1}=B^{p^{n-i-1}}E,\; HGH^{-1}=G^{-1},
$$
$$
DAD^{-1}=A^{1+p^{n-i-1}},\; DB=BD,\; DC=CD,
$$
$$
EAE^{-1}=AB^{p^{n-i-1}},\; EB=BE,\; EC=CE,
$$
$$
FA=AF,\; FBF^{-1}=A^{p^{n-i-1}}B,\; FC=CF,
$$
$$
GAG^{-1}=A^{g_0},\; GBG^{-1}=B^{h_0},\; GC=CG,
$$
$$
HAH^{-1}=B,\; HBH^{-1}=A,\; HCH^{-1}=C^{-1},\; HDH^{-1}=D C^{vp^{n-i-2}},
$$
where $v\in\Z$ satisfies $dv\equiv 1\mod p$, and the generators $\overline{D},\overline{E},\overline{F},\overline{G},\overline{H}$
of 
\begin{equation}
\label{otre}
\mathrm{Out}(\S)\cong \Z/p^{i}\Z\times [(\Z/p\Z\times \Z/p\Z)\rtimes (\Z/(p-1)\Z\rtimes \Z/2\Z)]
\end{equation}
satisfy the defining relations naturally arising from the above, and realizing (\ref{otre}). 

In addition, $\Pi=\langle A,B,C,D,E,F\rangle$ is a normal Sylow $p$-subgroup  of $\mathrm{Aut}(\S)$,
with the above defining relations, excluding those involving $G$ or $H$, and $\Pi$ is the kernel
of the natural homomorphism $\Gamma:\mathrm{Aut}(\S)\to \mathrm{Aut}(\S/\S^p)$.
Furthermore, $\Pi$ cannot be generated by fewer than 6 elements.
\end{theorem}

\begin{proof}
That the maps defining $D,E,F,G,H$ actually extend to automorphisms of $\S$ is a routine calculation
involving the defining relations (\ref{sima}) of $\S$ and will be omitted.

The stated relations among $A,B,C,D,E,F,G,H$ are easily verified. Moreover, any group generated
by elements $A$ through $H$ subject to the given relations has order $\leq 2p^{3n-2i+1}(p-1)$.

We claim that $\mathrm{Inn}(\S)\cap\langle D,E,F,G\rangle$
is trivial. Indeed, the defining relations of $\S$ yield
$$
A^j(\C)=\A^{-jep}\C,\; B^j(\C)=\B^{-jdp}\C,\; A^j(\B)=\B \D^j,\;
B^j(\A)=\A \D^{-j},\quad j\in\Z.
$$
A general element of $\mathrm{Inn}(\S)$ (resp. $\langle D,E,F,G\rangle$) has the form $A^jB^k C^m$ (resp.
$D^a G^f E^b F^c$) for some $j,k,m\in\Z$ (resp. $a,b,c,f\in\Z$). Suppose that
$A^jB^k C^m=D^a G^f E^b F^c$. We wish to show that this common element is trivial. On the one hand, we have
$$(A^jB^k C^m)(\C)=\B^{-kdp}\A^{-jep}\D^{-jkdp}\C
$$
and on the other hand
$$
(D^a G^f E^b F^c)(\C)=\C^{(1+p^{n-i-1})^a}.
$$
From the normal form of the elements of $\S$ deduce that $p^{n-i-1}|j$ and $p^{n-i-1}|k$. But then $\D^{-jkdp}=1$
and a fortiori
$$
\C=\C^{(1+p^{n-i-1})^a}.
$$
This forces $D^a=1$. Now
$$
(G^f E^b F^c)(\A)=\A^{g_0^f}\B^{h_0^fbp^{n-i-1}},
$$
while
$$
(A^jB^k C^m)(\A)=\A^{(1+ep)^m} \D^{-k(1+ep)^m}.
$$
This implies that $p|b$, $p^{n-i}|k$, and $(1+ep)^m\equiv g_0^f\mod p^{n-i}$. The first two conditions yield that
$E^b=1$ and $B^k=1$. As for the third, since the order of $(1+ep)^m$ modulo $p^{n-i}$ is a $p$-power and that of $g_0^f$
is a factor of $p-1$, both orders are equal to 1. This forces $G^f=1$ and $C^m=1$. Finally, we have
$$
F^c(\B)=\A^{cp^{n-i-1}}\B,
$$
while
$$
A^j(\B)=\B\D^j.
$$
Thus $p|c$ and $p^{n-i}|j$, whence $F^c=1$ and $A^j=1$. This proves the claim.

We show below that (\ref{autoabc}) holds. Since $\langle A,B,C,D,E,F,G\rangle\subseteq \Delta$ and $H\notin\Delta$,
we infer $$\Delta=\mathrm{Inn}(\S)\rtimes\langle D,E,F,G\rangle,\; \mathrm{Aut}(\S)=\Delta\rtimes\langle H\rangle,\; [\mathrm{Aut}(\S):\Delta]=2.$$
On the other hand, the very definitions of $D,E,F,G$ yield
$$\langle D,E,F,G\rangle\cong \Z/p^{i}\Z\times [(\Z/p\Z\times \Z/p\Z)\rtimes \Z/(p-1)\Z],$$ while Lemma \ref{centro} gives
$$\mathrm{Inn}(\S)\cong \S/Z(\S)\cong (\Z/p^{n-i}\Z\times \Z/p^{n-i}\Z)\rtimes \Z/p^{n-i-1}\Z.$$ Thus
$|\mathrm{Aut}(\S)|=2p^{3n-2i+1}(p-1)$, whence the given relations among $A,B,C,D,E,F,G,H$ must be defining relations,
and $\mathrm{Aut}(\S)=\langle A,B,C,D,E,F\rangle\rtimes (\langle G\rangle\rtimes\langle H\rangle)$ holds. This decomposition
of $\mathrm{Aut}(\S)$ yields (\ref{otre}), whence the given presentation of $\mathrm{Aut}(\S)$ immediately gives one for $\mathrm{Out}(\S)$.

That $\Pi$ is a normal Sylow $p$-subgroup  of $\mathrm{Aut}(\S)$
with the stated defining relations follows directly from above. The very definition of $A,B,C,D,E,F$ places
them in $\ker\Gamma$. Thus $\ker\Gamma=\Pi$ provided we show that $\ker\Gamma\cap (\langle G\rangle\rtimes\langle H\rangle)$
is trivial. To this end, let $l\in\Z$ and suppose, if possible, that $G^l H\in\ker\Gamma$. This implies $\C^2\in\S^p$, whence $\C\in\S^p$, against Proposition \ref{potp}.
Suppose next that $G^l\in\ker\Gamma$. This implies $\A^{{g_0}^l}\in\S^p$, so $\A^{{g_0}^l}=\A^k$ for some multiple $k$ of $p$,
by Proposition \ref{potp}. The order of this common element is a factor of $p-1$ as well as a $p$-power, whence this element is 1,
which implies $G^l=1$. This proves that $\ker\Gamma=\Pi$.
The defining relations of $\Pi$ make it clear that $[\Pi,\Pi]\subseteq \Pi^p$, which
yields an epimorphism $\Pi\to V$, where $V$ is 6-dimensional $\Z/\Z p$-vector space.
Thus $\Pi$ cannot be generated by fewer than 6 elements.

We proceed to show (\ref{autoabc}). Let $\Omega$ be an arbitrary automorphism of $\S$. By Proposition \ref{charac},
\begin{equation}\label{omet}
\Omega(\A)=\A^a\B^b\D^c,\;
\Omega(\B)=\A^f\B^j\D^k,\;
\Omega(\C)=\A^m\B^q\C^r,
\end{equation}
where $a,b,c,f,j,k,m,q,r\in\N$. Since $\Omega$ is an automorphism, we have
\begin{equation}\label{omeq}
{}^{\Omega(\C)}\Omega(\A)=\Omega({}^\C\A)=\Omega(\A)^{1+ep},\;
{}^{\Omega(\C)}\Omega(\B)=\Omega({}^\C\B)=\Omega(\B)^{1+dp}.
\end{equation}
Using (\ref{omet}) and comparing the $\A$-components of each side of (\ref{omeq}) yields
\begin{equation}\label{jorg}
\A^{(1+ep)^r a}=\A^{(1+ep) a},\; \A^{(1+ep)^r f}=\A^{(1+dp) f}.
\end{equation}
By Proposition \ref{charac}, $\Omega$ induces an automorphism
in $\T/Z(\T)$. Let $M\in\GL_2(\Z/p^{n-i}\Z)$ be the matrix of this automorphism with the respect to the basis formed by
$\A,\B$ when taken modulo $Z(\T)$. Then (\ref{omet}) gives
$$
M=\left(\begin{array}{cc} \bar{a} & \bar{b}\\ \bar{f} & \bar{j}\end{array}\right),
$$
where the bar indicates the passage from $\Z$ to $\Z/p^{n-i}\Z$. Since $M$ is invertible, one of $a,f$ is not divisible by $p$, whence
by (\ref{jorg})
$$
(1+ep)^r \equiv 1+ep \mod p^{n-i}\text{ or }(1+ep)^r \equiv 1+dp \mod p^{n-i}.
$$
Thus by (\ref{eqrd2})
$$
(1+ep)^{r-1} \equiv 1\mod p^{n-i}\text{ or }(1+ep)^{r+1} \equiv 1 \mod p^{n-i}.
$$
Since $p\nmid e$, Lemma \ref{uno} ensures that $1+ep$ has order $p^{n-i-1}$ modulo $p^n$, whence
$$
r\equiv 1\mod p^{n-i-1}\text{ or }r\equiv -1\mod p^{n-i-1}.
$$
In the second case, we replace $\Omega$ by $H\Omega$, so we may assume that the first case occurs. Thus
\begin{equation}\label{req2}
r\equiv 1\mod p^{n-i-1}.
\end{equation}
Taking into account (\ref{eqrd2}), we deduce
\begin{equation}\label{req}
(1+ep)^r \equiv 1+ep \mod p^{n-i}\text{ and }(1+dp)^r \equiv 1+dp \mod p^{n-i}.
\end{equation}

In view of (\ref{req}), the automorphisms that
conjugation by $\C^r$ and $\C$ induce on $\T/Z(\T)$ are identical. Let $N$ be the matrix of this automorphism
relative to the same basis used above. Then
$$
N=\left(\begin{array}{cc} \overline{1+ep} & \bar{0}\\ \bar{0} & \overline{1+dp}\end{array}\right).
$$
Using once more that $\Omega$ is an automorphism of $\S$, we see that for any $t\in \T$,  we have
\begin{equation}\label{pat}
{}^{\Omega(\C)}\Omega(t)=\Omega({}^\C t).
\end{equation}
Since conjugation by $\C^r$ and $\C$ induce the same automorphism on $\T/Z(\T)$, (\ref{omet}) and (\ref{pat}) give
$$
{}^{\C}\Omega(t)\, Z(\T)=\Omega({}^\C t)\, Z(\T).
$$
In matrix terms, this means that
$$
NM=MN,
$$
which is equivalent to
\begin{equation}\label{pep}
b(e-d)p\equiv 0\mod p^{n-i},\; f(e-d)p\equiv 0\mod p^{n-i}.
\end{equation}
It follows from (\ref{enod}) and (\ref{pep}) that
\begin{equation}\label{eqbf}
b\equiv 0\equiv f\mod p^{n-i-1}.
\end{equation}
We now use for the first time the hypothesis $i<n-1$. Since $M$ is invertible, (\ref{eqbf}) implies
\begin{equation}\label{eqaj}
\overline{a},\overline{j}\in (\Z/p^{n-i}\Z)^\times.
\end{equation}

Using (\ref{omet}) and (\ref{req}), and carefully comparing the $\D$-components of each side of (\ref{omeq}) gives
\begin{equation}\label{eqcc}
\D^{-(1+ep)aq}\D^{(1+dp)bm}\D^{c}=\D^{-abep(ep+1)/2}\D^{c(1+ep)},\;
\D^{-(1+ep)fq}\D^{(1+dp)jm}\D^{k}=\D^{-fjdp(dp+1)/2}\D^{k(1+dp)}.
\end{equation}
We claim that
\begin{equation}\label{equ}
q\equiv 0\equiv m\mod p.
\end{equation}
Indeed, using (\ref{eqbf}), that $n-1>i$, and that $e(ep+1)$ (resp. $d(dp+1)$) is even, we deduce from the first
(resp. second) equality in (\ref{eqcc}) that $\D^{aq}$ (resp.  $\D^{jm}$) is in $\langle \D^p\rangle$.
Thus (\ref{eqaj}) yields $q\equiv 0\mod p$ (resp. $m\equiv 0\mod p$). This proves the claim.

On the other hand, in view of (\ref{req}), we have
\begin{equation}\label{oja}
\A\C^r\A^{-1}=\A^{-ep}\C^r,\; \B\C^r\B^{-1}=\B^{-dp}\C^r.
\end{equation}
We deduce from (\ref{omet}), (\ref{equ}), and (\ref{oja}) that if we replace $\Omega$ by $\delta_s\Omega$,
where $\delta_s\in\mathrm{Inn}(\S)$ for a suitable $s\in \T$, the following will hold in (\ref{omet}):
\begin{equation}\label{qmcero}
q=0=m.
\end{equation}

Making use of (\ref{req2}) and replacing $\Omega$ by $D^l\Omega$ for a suitable $l\in\Z$, we will have
$q=0=m$ and $r=1$ in (\ref{omet}), so that
\begin{equation}\label{eqnu}
\Omega(\C)=\C.
\end{equation}
Taking into account that
$$
\D^{aj-bf}=[\Omega(\A),\Omega(\B)]=\Omega([\A,\B])=\Omega(\D)=\D,
$$
and appealing to (\ref{eqbf}), we infer
\begin{equation}\label{eqajota}
aj\equiv 1\mod p^{n-i}.
\end{equation}
Going back to (\ref{eqcc}) and making use of (\ref{qmcero}), we see that
$$
\D^{abep(ep+1)/2}=\D^{cep},\;
\D^{fjdp(dp+1)/2}=\D^{kdp}.
$$
But $d(dp+1)$ and $e(ep+1)$ are even, so (\ref{eqbf}) yields
$$
\D^{cep}=1=\D^{kdp}.
$$
Since $d$ and $e$ are not divisible by $p$, we deduce
\begin{equation}\label{eqcks}
c\equiv 0\equiv k\mod p^{n-i-1}.
\end{equation}
Now $\A^{p^{n-i-1}}$ and $\B^{p^{n-i-1}}$ commute with $\C$ and with each other, so replacing $\Omega$ by
$\delta_s\Omega$, where $\delta_s\in\mathrm{Inn}(\S)$ for a suitable $s\in \T$,
(\ref{eqbf}) and (\ref{eqcks}) imply that (\ref{eqnu}) still holds and we further have
$$
c=0=k
$$
in (\ref{omet}).

We next replace $\Omega$ by $\delta_s G^l\Omega$, for a suitable $l\in\Z$ and $\delta_s\in\mathrm{Inn}(\S)$, with
$s\in \langle \C\rangle$.
This will keep (\ref{eqnu}) and $c=0=k$, and, because of (\ref{eqajota}), will further achieve $a=1=j$ in (\ref{omet}).
Taking into account~(\ref{eqbf}), we see that $E^uF^v\Omega=1$ for suitable $u,v\in\Z$. This proves (\ref{autoabc}),
completing the proof of the theorem.
\end{proof}


\begin{prop}\label{kerla} Let $\Lambda:\mathrm{Aut}(\S)\to \mathrm{Aut}(\T)$ be the restriction homomorphism ensured by
Proposition \ref{charac}. Then $\ker\Lambda$ is the cyclic group of order $p^{i-1}$ generated by
$$
\A \mapsto \A,\;\B\mapsto\B,\; \C\mapsto \C^{1+p^{n-i}}.
$$
\end{prop}

\begin{proof} Let $\Omega\in \ker\Lambda$. Then $\Omega(\C)=t\C^s$, where $t\in \T$ and $s\in\N$.
Since $\Omega$ is the identity on $\T$,
$$
\A^{1+ep}=\Omega(\A^{1+ep})=\Omega(\C\A\C^{-1})=\Omega(\C)\Omega(\A)\Omega(\C^{-1})=
t\C^s\A\C^{-s}t^{-1}=t\A^{(1+ep)^s}t^{-1}.
$$
Since conjugation by $t$ sends $\A$ to itself times some power of $\D$, this forces
$$
(1+ep)^s\equiv (1+ep)\mod p^{n-i},
$$
so
$$
(1+ep)^{s-1}\equiv 1\mod p^{n-i}.
$$
As $p\nmid e$, Lemma \ref{uno} yields
$$
s\equiv 1\mod p^{n-i-1}.
$$
We also deduce that $t$ must commute with $\A^{1+ep}$
and hence with $\A$. A like calculation with $\B$ instead of $\A$ forces $t$ to commute with $\B$,
whence $t\in Z(\T)=\langle\D\rangle$. Since $\D=\C^{p^{i-1}}$ and $i-1\geq n-i-1$, we infer
$\Omega(\C)=\C^k$, where $k\equiv 1\mod p^{n-i-1}$. But $\Omega$ must fix $\D$ as well, so in fact $k\equiv 1\mod p^{n-i}$.
\end{proof}


Note that if $i\neq n-1$ then, in the notation of Theorem \ref{mainresult}, we have $\ker\Lambda=\langle D^p\rangle$.




\begin{theorem}\label{lindop} Suppose $i=n-1$, so that $\S$ be the group of order $p^{n+1}$ generated by $\A,\B,\C$, with defining relations
$$
\A^p=1,\B^p=1,\C^{p^{n-1}}=1, [\A,\B]=\C^{p^{n-2}}, \A\C=\C\A, \B\C=\C\B,
$$
and $\T\cong\Heis(\Z/p\Z)$ is the subgroup of $\S$ of order~$p^3$ generated by
$\A,\B$ and $\C^{p^{n-2}}=\D$.

Let $\Lambda:\mathrm{Aut}(\S)\to \mathrm{Aut}(\T)$ be the restriction homomorphism whose existence is ensured by
Proposition \ref{charac}. Then
$$
\mathrm{Aut}(\S)\cong\ker\Lambda\times \mathrm{Im}\Lambda,
$$
where $\ker\Lambda$ is a cyclic group of order $p^{n-2}$ and $\mathrm{Im}\Lambda=\mathrm{Aut}(\T)$ is an extension
of $\Z/p\Z\oplus\Z/p\Z$ by $\GL_2(\Z/p\Z)$. More precisely, let
$A,B\in\mathrm{Inn}(\S)$ be the inner automorphisms of $\S$ associated with $\A,\B$, respectively.
Let $C,D,E,F,G\in \mathrm{Aut}(\S)$ be respectively defined by
$$
\A\mapsto\A,\; \B\mapsto\A\B,\; \C\mapsto \C,
$$
$$
\A\mapsto\B,\; \B\mapsto\A^{-1},\; \C\mapsto \C,
$$
$$
\A\mapsto\A^r,\; \B\mapsto\B^{s},\; \C\mapsto \C,
$$
$$
\A\mapsto\A^r,\; \B\mapsto\B,\; \C\mapsto \C^r,
$$
$$
\A \mapsto \A,\;\B\mapsto\B,\; \C\mapsto \C^{1+p},
$$
where $r$ is any integer of order $p-1$ modulo $p^{n-1}$ (and hence modulo $p$), with inverse $s$ modulo~$p^{n-1}$,
(and hence modulo $p$), both of which are taken to be odd (as we are allowed to do), and further let $t=r(r-1)/2$.
Then $\ker\Lambda=\langle G\rangle$ and $\mathrm{Aut}(\S)$ is generated by $A,B,C,D,E,F,G$, with defining relations
$$
A^p=1,B^p=1,AB=BA,
$$
$$
C^p=1,D^2=E^{(p-1)/2}, E^{p-1}=1, F^{p-1}=1, G^{p^{n-2}}=1, DED^{-1}=E^{-1}, DFD^{-1}=E^{-1}F,
$$
$$
EF=FE, ECE^{-1}=A^{-t}C^{r^2}, FCF^{-1}=C^r,
$$
$$
A^{(r^k-1)/2} B^{-(s^k+1)/2} DC^{s^k} D= E^{k} C^{-s^k}DC^{-r^k},\quad 0\leq k<p-1,
$$
$$
CAC^{-1}=A, CBC^{-1}=AB, DAD^{-1}=B, DBD^{-1}=A^{-1}, EAE^{-1}=A^r,
$$
$$
EBE^{-1}=B^s, FAF^{-1}=A^r, FBF^{-1}=B,
$$
$$
AG=GA, BG=GB, CG=GC, DG=GD, EG=GE, FG=GF.
$$
Moreover, $\mathrm{Aut}(\T)$ is generated by the restrictions to $\T$ of $A$ through $F$, say $A_0$ through~$F_0$,
respectively, with the above defining relations, mutatis mutandi (this means that we replace $A$ through $F$ by $A_0$ through $F_0$
and we remove all relations involving $G$). Furthermore,
$$
|\mathrm{Aut}(\S)|=p^{n+1}(p-1)^2(p+1),\; |\mathrm{Aut}(\T)|=p^3(p-1)^2(p+1).
$$
\end{theorem}

\begin{proof} That the maps defining $A$ through $G$ extend to automorphisms of $\S$ follows from the defining relations of $\S$.

We have $\ker\Lambda=\langle G\rangle$, a group of order $p^{n-2}$, by Proposition \ref{kerla}.
Let $Z$ be the subgroup of $\mathrm{Aut}(\T)$  generated by $A_0$ through $F_0$. We claim that $Z=\mathrm{Aut}(\T)$.
Indeed, we have a homomorphism $\Gamma:\mathrm{Aut}(\T)\to \mathrm{Aut}(\T/Z(\T))$ whose kernel is easily seen to be
$\mathrm{Inn}(\T)$. Relative to the basis $\{\A Z(\T),\B Z(\T)\}$ of the $\Z/p\Z$-vector space
$\T/Z(\T)$, the
matrices of $\Gamma(C_0)$ through $\Gamma(F_0)$ are respectively given by
$$
\left(\begin{array}{cc} 1 & 1\\ 0 & 1\end{array}\right), \left(\begin{array}{cc} 0 & -1\\ 1 & 0\end{array}\right),
\left(\begin{array}{cc} r & 0\\ 0 & r^{-1}\end{array}\right), \left(\begin{array}{cc} r & 0\\ 0 & 1\end{array}\right).
$$
As these matrices generate $\GL_2(\Z/p\Z)$, we see that the restriction of $\Gamma$ to $Z$ is surjective. Since
$\mathrm{Inn}(\T)\subset Z$, it follows that $Z=\mathrm{Aut}(\T)$. We have shown, in particular, that
$\mathrm{Aut}(\T)$ is an extension of $\ker\Gamma\cong\Z/p\Z\oplus\Z/p\Z$ by $\mathrm{Im}\Gamma\cong\GL_2(\Z/p\Z)$.

The verification of the given relations among $A$ through $G$ is a routine calculation and will be omitted. We next
show that these are defining relations.

Let $P$ be the abstract group generated by elements $A$ through $G$ subject to the given relations. It suffices to show
that $|P|\leq p^{n+1}(p-1)^2(p+1)$. In this respect, it is clear that $\langle A,B,G\rangle$ is a normal subgroup of $P$
of order $\leq p^n$. Let $Q=P/\langle A,B,G\rangle$ and let $c,u,v,f$ be the images of $C,D,E,F$,  respectively,
under the canonical projection $P\to Q$. We further let $N=\langle c,v,f\rangle$. Then $v,f$ commute and normalize
$\langle c\rangle$, so $|N|\leq p(p-1)^2$. We claim that $Q=N\cup NuN$. This means that $Y=N\cup NuN$ is a subgroup of $Q$.
Now $u^{-1}=u^3=uu^2$, with $u^2\in N$, so $Y$ is closed under inversion. That $Y$ is closed
under multiplication is respectively equivalent to $NuNuN\subset Y$, $uNu\subset Y$, and $uNu^{-1}\subset Y$. Now
$u$ conjugates $v,f$ back into $N$, and $u\langle c\rangle u\subset Y$ by the given relations, so $Y$ is closed
under multiplication. Since $NuN=Nu\langle c\rangle$, we have
$$
|NuN|\leq |N||\langle c\rangle|=p^2(p-1)^2,
$$
and therefore
$$
|Q|\leq |N|+|NuN|\leq p(p-1)^2+p^2(p-1)^2=p(p-1)^2(p+1),
$$
whence
$$
|P|=|\langle A,B,G\rangle||Q|\leq p^n p(p-1)^2(p+1)=p^{n+1}(p-1)^2(p+1),
$$
as required.  This shows that the stated relations are defining relations for $\mathrm{Aut}(\S)$.

Now, $F$ normalizes $\langle A,B,C,D,E\rangle$, every element of this group fixes $\C$,
while $F(\C)=\C^r$ and $G(\C)=\C^{1+p}$. Since $r$ has order $p-1$ modulo $p^{n-1}$ and $1+p$ has order $p^{n-2}$
modulo $p^{n-1}$ (where $p^{n-1}$ is the order of $\C$), it follows that the intersection of
 $\langle A,B,C,D,E,F\rangle=\langle A,B,C,D,E\rangle\rtimes \langle F\rangle$ with $\langle G\rangle$
is trivial, so the restriction of $\Lambda$ to $\langle A,B,C,D,E,F\rangle$ is an isomorphism and, since $G$
is central, we have $\mathrm{Aut}(\S)=\langle G\rangle\times \langle A,B,C,D,E,F\rangle$.
\end{proof}


\section{The automorphism group of $\S$ when $2i<n$}\label{rojo2}

We assume throughout this section that $2i<n$.

\begin{lemma}\label{centro2i} The center of $\S$ is generated by
$\D$ and $\B^{p^{n-i-1}}$.
\end{lemma}

\begin{proof} Let $z\in Z(\S)$. Then $z=\C^r t$, where $t\in \T$. The first part of the proof of Lemma~\ref{centro}
shows that $t\in Z(\T)$, that $\C^r$ commutes with $\A$ and $\B$, and that $p^{n-i-1}|r$. In this case,
$Z(\T)=\langle \D,\B^{p^i} \rangle$.  As $\D=\C^{p^{n-i-1}}$, we infer
$Z(\S)\subseteq \langle \B^{p^i}, \D \rangle$.  Here $\D$ actually belongs to $Z(\S)$,
so $Z(\S)=\langle \B^{p^j}, \D \rangle$, where $j\geq i$ is large enough so that $\C$
commutes with $\B^{p^j}$, that is, $j\geq n-i-1$. Since $n-i-1\geq i$, the proof is complete.
\end{proof}

\begin{lemma}\label{sumalinda} Let $a,b\in\N$ and $z\in\Z$ be such that
\begin{equation}\label{hipo}
z^{p^a}\equiv 1\mod p^b.
\end{equation}
Then
$$
1+z+z^2+\cdots+z^{p^a-1}\equiv p^a\mod p^b.
$$
\end{lemma}

\begin{proof} It follows from (\ref{hipo}) that $z\equiv 1\mod p$. Indeed, the order of $z$ modulo $p$ must be a factor of
both $p^a$ and $p-1$ (which is the order of $(\Z/p\Z)^\times$), so $z$ has order 1 modulo $p$.

If $z=1$ the result is obvious. If not, $z=1+hp^s$, where $h\in\Z$ not divisible by $p$ and $s\geq 1$. Then by Lemma \ref{uno},
$$
z^{p^a}\equiv  (1+hp^s)^{p^a}\equiv 1+hp^{s+a}\mod p^{2s+a},
$$
so for some integer $f$, we have
$$
z^{p^a}=1+hp^{s+a}+fp^{2s+a}=1+p^{s+a}(h+fp^s).
$$
Here $p\nmid (h+fp^s)$, because $s\geq 1$. It follows from (\ref{hipo}) that $b\leq s+a$.
Now
$$
1+z+z^2+\cdots+z^{p^a-1}\equiv (z^{p^a}-1)/(z-1)\equiv p^{s+a}(h+fp^s)/hp^s\mod p^b.
$$
Here  $z-1$ divides $z^{p^a}-1$, so $hp^s$ divides $p^{s+a}(h+fp^s)$. But $hp^s$ also divides $hp^{s+a}$, so
$hp^s$ divides $p^{s+a}fp^s$, that is $h$ divides $p^{s+a}f$. Since $h$ is relatively prime to $p$, we infer that $h$ divides $f$. As
$b\leq s+a$, we conclude that
$p^{s+a}(h+fp^s)/hp^s\equiv p^a+fp^{s+a}/h\equiv p^a\mod p^b$.
\end{proof}

\begin{theorem}\label{main2} Let $p$ be an odd prime. Suppose $i,n\in\N$ satisfy
$n\geq 2$, $1\leq i\leq n-1$, and $2i<n$. Let $d,e\in\Z$ be chosen so that
(\ref{eqrd}) and (\ref{eqrd2}) hold, and let $\S$ be the group with presentation (\ref{sime}).

Let $L,M,N\in\mathrm{Inn}(\S)$ be the inner automorphisms of $\S$ associated with $\A,\B,\C$, respectively.
Let $U,V,W,X,Y,Z\in \mathrm{Aut}(\S)$ be respectively the defined by
$$
\A \mapsto \A,\; \B\mapsto\B\C^{p^{n-i-2}},\; \C\mapsto \A^{-e} \C,
$$
$$
\A \mapsto \A,\; \B\mapsto\B^{1+p^{n-i-1}},\ \C\mapsto \C^{1+p^{n-i-1}},
$$
$$
\A\mapsto\A\B^{p^{n-i-1}},\; \B\mapsto\B,\; \C\mapsto\C,
$$
$$
\A\mapsto\A,\; \B\mapsto\A^{p^{i-1}}\B,\; \C\mapsto\C,
$$
$$
\A\mapsto\A\C^{p^{n-2}},\; \B\mapsto\B,\; \C\mapsto\C,
$$
$$
\A\mapsto\A^{g_0},\; \B\mapsto\B^{h_0},\; \C\mapsto\C,
$$
where $g_0$ is any integer of order $p-1$ modulo $p^{n-i}$, with inverse $h_0$ modulo $p^{n-i}$. Let $G$ be the subgroup of
$\mathrm{Aut}(\S)$ preserving $\T$. Then
\begin{equation}
\label{autolmn}
G=\langle L,M,N,V,W,X,Y,Z\rangle
\end{equation}
and
$$
\mathrm{Aut}(\S)=G\langle U\rangle=\langle U\rangle G.
$$
\end{theorem}

\begin{proof} Let $\Omega\in \mathrm{Aut}(\S)$. Since $\A^{p^i}=1$, Proposition \ref{charac} ensures
that $\Omega(\A)\in \T$, so $\Omega(\A)=PQR$ for some $P\in\langle\A\rangle$,
$Q\in\langle\B\rangle$, and $R\in\langle\D\rangle$.
Now the order of $\B$ is $p^{n-i}$, with $n-i>i$, but equation (\ref{luca2})
from Proposition \ref{charac} still gives $\Omega(\B)=t_0\C_0$, where $t_0\in \T$ and
$\C_0\in\langle \C^{p^{n-i-2}}\rangle$. Using the normal form of elements of $\T$, it follows that
$\Omega(\B)=AB\C^{sp^{n-i-2}}$ for some $A\in\langle\A\rangle$,
$B\in\langle\B\rangle$, and $s\in\N$. As $\Omega(\C)\in\S$, we have $\Omega(\C)=CD\C^{r}$
for some $C\in\langle\A\rangle$,
$D\in\langle\B\rangle$, and $r\in\N$. Thus
\begin{equation}
\label{abcd}
\Omega(\B)=AB\C^{sp^{n-i-2}},\;  \Omega(\A)=PQR,\;
\Omega(\C)=CD\C^{r}.
\end{equation}
Since $\Omega(\B),\Omega(\A), \Omega(\C)$ must generate $\S$, we see that $p\nmid r$.

\smallskip

\noindent{\sc Claim 1.} $Q\in \langle \B^{p^{n-2i}}\rangle$.

\smallskip

Since $1=\Omega(\A^{p^i})=(PQR)^{p^i}=P^{p^i}Q^{p^i}R^u$, $u\in\Z$, and the order of $\B$ is $p^{n-i}$, the result follows.

\smallskip

\noindent{\sc Claim 2.} $P$ generates $\langle\A\rangle$, $B$ generates $\langle\B\rangle$ and
$\Omega(\D)$ generates $\langle \D\rangle$.

\smallskip

Set $f=p^{n-i-2}s$. We then have
$$
\Omega(\D)=\Omega([\A,\B])=[\Omega(\A),\Omega(\B)]=[PQR,AB\C^{f}].
$$
Here
$$
[PQR,AB\C^{f}]=[PQR,A]\cdot\,{}^A[PQR,B\C^{f}]=[PQR,A]\cdot\,{}^A[PQR,B]\cdot\, {}^{AB} [PQR,\C^{f}],
$$
where
$$
[PQR,A]\cdot\,{}^A[PQR,B]=[Q,A][P,B],
$$
$$
[\C^{f},PQR]=[\C^{f},P]\cdot\, {}^{P}[\C^{f},QR]=
[\C^{f},P]\cdot\,{}^{P}[\C^{f},Q].
$$
Using Lemma \ref{uno}, $P^{p^i}=1$, and $n-i-1\geq i$, we see that
$$
[\C^{f},P]=P^{(1+ep)^{f}}P^{-1}=P^{1+sep^{n-i-1}}P^{-1}=1.
$$
Moreover, by Claim 1, $Q^{p^i}=1$. Using again Lemma \ref{uno} and $n-i-1\geq i$, we obtain
$$
[\C^{f},Q]=Q^{(1+dp)^{f}}Q^{-1}=Q^{1+sdp^{n-i-1}}Q^{-1}=1.
$$
Therefore,
$$
\Omega(\D)=[PQR,AB\C^{f}]=[Q,A][P,B]\in \langle \D\rangle.
$$
This shows that $\langle \D\rangle$ is a characteristic subgroup of $\S$, whence
$\Omega(\D)$ generates $\langle \D\rangle$. In particular, $[Q,A][P,B]$ has order $p^i$.
But $Q\in\langle\B^{p^{n-2i}}\rangle$ by Claim 1, so
$$
[Q,A]\in \langle\D^{p^{n-2i}}\rangle
$$
has order $<p^i$. Since $[P,B]$ commutes with $[Q,A]$, it follows that $[P,B]$ has order $p^i$. Thus $\langle P\rangle=\langle\A\rangle$ and $\langle B\rangle=\langle\B\rangle$.

\smallskip

\noindent{\sc Claim 3.} $r\equiv 1\mod p^{n-i-1}$.

\smallskip

We have ${}^{\Omega(\C)}\Omega(\B)=\Omega(\B)^{1+dp}$. Setting $f=p^{n-i-2}s$, this translates into
$$
{}^{CD \C^r} (AB\C^f)=(AB\C^f)^{1+dp}.
$$
The $\B$-component of the left hand side is equal to $B^{(1+dp)^r}$.
As for the $\B$-component of the right hand side, we have
$$
(AB\C^f)^p=ABA^{(1+ep)^f}B^{(1+dp)^f}\cdots A^{(1+ep)^{f(p-1)}}B^{(1+dp)^{f(p-1)}}\D^s.
$$
Using Lemma \ref{sumalinda} with $z=(1+dp)^f$, $a=1$ and $b=n-i$ we see that the $\B$-component of $(AB\C^f)^p$ is $B^p$
and that the $\A$-component of $(AB\C^f)^p$ is $A^p$.
Reordering the factors, it follows that
$$
(AB\C^f)^p=A^pB^p\D^u,\quad u\in\Z.
$$
Thus
$$
(AB\C^f)^{dp}=A^{pd} B^{pd}\D^q, \quad q\in\Z.
$$
so
$$
(AB\C^f)^{1+dp}=(AB\C^f)(AB\C^f)^{dp}=AB\C^f A^{pd} B^{pd}\D^q,
$$
and therefore
$$
(AB\C^f)^{1+dp}=AA^{(1+ep)^f pd}BB^{(1+dp)^f pd}\C^f \D^l, \quad l\in\Z.
$$
Now by Lemma \ref{uno},
$$
(1+dp)^f pd\equiv pd\mod p^{n-i},
$$
so the $\B$-component of $(AB\C^f)^{1+dp}$ is $B^{1+dp}$. All in all, we deduce
$$
B^{(1+dp)^r}=B^{1+dp}.
$$
Since $B$ has order $p^{n-i}$ by Claim 2, we infer from $p\nmid d$ and Lemma \ref{uno} that
$$
r\equiv 1\mod p^{n-i-1}.
$$

\medskip

\noindent{\sc Claim 4.} The map defining $U$ extends to an automorphism of $\S$.

\medskip

It is clear that the images of $\C,\A,\B$ generate $\S$. We need to verify that the defining relations of
$\S$ are preserved. We proceed to do this, labeling each step with the corresponding relation of $\S$.

$\bullet$ $\B^{p^{n-i}}=1$. Set $f=p^{n-i-2}$ and $z=(1+dp)^f$. Then
$$
(\B \C^f)^p=\B^u \D,
$$
where
$$
u=1+z+\cdots+z^{p-1}\mod p^{n-i}.
$$
But $z^p\equiv 1\mod p^{n-i}$, so Lemma \ref{sumalinda} implies $u\equiv p\mod p^{n-1}$, whence
\begin{equation}
\label{otraver2}
(\B \C^f)^p=\B^p \D.
\end{equation}
This element has order $p^{n-i-1}$, since the factors commute, $\D$ has order $p^i$, and $n-i-1\geq i$.

$\bullet$ $\A^{p^i}=1$. This step is clear.

$\bullet$ $[\A,\B]=\C^{p^{n-i-1}}$. We need to verify that
$$
[\A,\B\C^{p^{n-i-2}}]=(\A^{-e} \C)^{p^{n-i-1}}.
$$
On the one hand, we have
$$
[\A,\B\C^{p^{n-i-2}}]=[\A,\B]\cdot\, {}^{\B}\![\A,\C^{p^{n-i-2}}]=[\A,\B]=\C^{p^{n-i-1}},
$$
because $\C^{p^{n-i-2}}$ commutes with $\A$, since $(1+ep)^{p^{n-i-2}}\equiv 1+ep^{n-i-1}\mod p^{n-i}$,
$n-i-1\geq i$, and $\A$ has order $p^i$. On the other hand,
\begin{equation}
\label{pin}
(\A^{-e} \C)^{p^{n-i-1}}=\C^{p^{n-i-1}}
\end{equation}
applying Lemma \ref{sumalinda} with $z=1+ep$, $a=n-i-1$, and $b=n-i$, and using $n-i-1\geq i$.

$\bullet$ $\C^{p^{n-1}}=1$. This step follows from (\ref{pin}).

$\bullet$ ${}^\C \A=\A^{1+ep}$. This step is clear because $\A^{-e}$ commutes with $\A$.

$\bullet$ ${}^\C \B=\B^{1+dp}$. Set $f=p^{n-i-2}$. We need to verify that
\begin{equation}
\label{otraver}
{}^{\A^{-e}\C} (\B \C^f)=(\B \C^f)^{1+dp}.
\end{equation}
Using (\ref{otraver2}) and the fact that $\D$ is central in $\S$, we have
$$
(\B \C^f)^{dp}=\B^{dp} \D^{d},
$$
so
$$
(\B \C^f)^{1+dp}=(\B \C^f)(\B \C^f)^{dp}=\B \C^f\B^{dp} \D^{d}.
$$
Here $\C^f$ and $\B^{dp}$ commute, since $dp(1+dp)^{p^{n-i-2}}\equiv dp\mod p^{n-i}$ by Lemma \ref{uno}. Thus
$$
(\B \C^f)^{1+dp}=\B^{1+dp}\C^f \D^{d}.
$$
On the other hand, since $n-i-1\geq i$, $\C^f$ commutes with $\A$, so
$$
{}^{\A^{-e}\C} (\B \C^f)=\,{}^{\A^{-e}}\!(\B^{1+dp} \C^f)=\B^{1+dp}\D^{-e(1+dp)}\C^f.
$$
Now from $(1+dp)(1+ep)\equiv 1\mod p^n$ we infer $e+d+edp\equiv 0\mod p^{n-1}$, so $d\equiv -e(1+dp)\equiv p^i$.
This completes the proof of Claim 4.

\medskip

Recall that $\langle B\rangle=\B$ by Claim 2. Using
(\ref{abcd}), we see that $U^l\Omega$ preserves $\T$ for a suitable $l$.
This proves that $\mathrm{Aut}(\S)=\langle U \rangle G$. Since this is a subgroup, it follows that
$\langle U \rangle G=G\langle U \rangle$.

We proceed to prove (\ref{autolmn}). The verification that the maps defining $V,W,X,Y,Z$ extend
to automorphisms of $\S$ is routine and will be omitted.

Let $\Omega\in G$. Then
\begin{equation}
\label{abce}
\Omega(\B)=ABS,\;  \Omega(\A)=PQR,\;
\Omega(\C)=CD\C^{r}.
\end{equation}
Here  $A,P,C\in\langle\A\rangle$, $B,Q,D\in\langle\B\rangle$, and $R,S\in\langle\D\rangle$.
We know from Claim 3 that $r\equiv 1\mod p^{n-i-1}$. Thus, replacing $\Omega$ by $V^j\Omega$ for a suitable $j$,
we will have $r=1$ in (\ref{abce}).


For any $t\in \T$,  we have
$$
{}^{\Omega(\C)}\Omega(t)=\Omega({}^\C t).
$$
Taking $t=\A$, we get
\begin{equation}
\label{boro}
{}^{(CD\C)}PQR={}^{\Omega(\C)}\Omega(\A)=\Omega({}^\C \A)=\Omega(\A^{1+ep})=\Omega(\A)^{1+ep}=(PQR)^{1+ep}.
\end{equation}
Comparing the $\B$-components in (\ref{boro}), we see that
$$
Q^{1+dp}=Q^{1+ep}.
$$
It follows from (\ref{enod}) that $Q^p=1$ and therefore
\begin{equation}
\label{qato}
Q\in\langle \B^{p^{n-i-1}}\rangle.
\end{equation}
Comparing the $\D$-components in (\ref{boro}), we find that
\begin{equation}
\label{tig}
[C,Q]^{1+dp}[D,P]^{1+ep}R=[Q,P][Q,P]^2\cdots[Q,P]^{ep}R^{1+ep}.
\end{equation}
This implies that
$$
[C,Q][D,P]\in \langle\D^{p}\rangle.
$$
Now $n-i-1\geq 1$ because $2i<n$, so (\ref{qato}) implies
$[C,Q]\in\langle \D^{p}\rangle$, whence $[D,P]\in \langle\D^{p}\rangle$.
Since $P$ generates $\langle\A \rangle$, we infer
\begin{equation}
\label{boto}
D\in \langle\B^p\rangle.
\end{equation}
Likewise, taking $t=\B$, we must have
\begin{equation}
\label{boro2}
{}^{(CD\C)}ABS={}^{\Omega(\C)}\Omega(\B)=\Omega({}^\C \B)=\Omega(\B^{1+dp})=\Omega(\B)^{1+dp}=(ABS)^{1+dp}.
\end{equation}
Comparing the $\A$-components in (\ref{boro2}), we see that
$$
A^{1+ep}=A^{1+dp}.
$$
Since $d\not\equiv e\mod p$, this forces $A^p=1$, whence
\begin{equation}
\label{ato}
A\in\langle \A^{p^{i-1}}\rangle.
\end{equation}
Comparing the $\D$-components in (\ref{boro2}), we see that
\begin{equation}
\label{tig2}
[C,B]^{1+dp}[D,A]^{1+ep}S=[B,A][B,A]^2\cdots[B,A]^{dp}S^{1+dp}.
\end{equation}
It follows that
$$
[C,B][D,A]\in\langle \D^{p}\rangle.
$$
In view of (\ref{boto}), we have $[D,A]\in\langle \D^{p}\rangle$.
Since $B$ generates $\langle\B \rangle$, we infer
\begin{equation}
\label{boton}
C\in \langle\A^p\rangle.
\end{equation}

Making use of (\ref{boto}) and (\ref{boton}) we see that if we replace $\Omega$ by $\delta_s\Omega$, where
$\delta_s\in\mathrm{Inn}(\S)$ for a suitable $s\in \T$, we will have $C=1=D$ in (\ref{abce}).
Since $r\equiv 1\mod p^{n-i-1}$, a further replacement of $\Omega$ by $V^l\Omega$ for a suitable $l$
will maintain $C=1=D$ and ensure $r=1$ in (\ref{abce}). We then have
\begin{equation}
\label{puton}
\D=\Omega(\D)=\Omega([\A,\B])=
[PQR,ABS]=[P,B][Q,A]=[P,B].
\end{equation}
Here $[Q,A]=1$ because (\ref{qato}) and (\ref{ato}) ensure that $[Q,A]\in\langle \C^{p^{(n-i-1)+(n-i-1)+(i-1)}}\rangle$,
and $n-i-1\geq 1$ as just noted above. On the other hand, we have
$$
P=\A^a,\; B=\B^j,
$$
where $a,j$ are relatively prime to $p$ by Claim 2. By virtue of (\ref{puton}) we obtain the sharper statement
\begin{equation}
\label{raton}
aj\equiv 1\mod p^i.
\end{equation}
Going back to (\ref{tig2}) with $C=1=D$ and using $A\in\langle \A^{p^{i-1}}\rangle$, we deduce $S^{dp}=1$
and hence
$$
S\in\langle\D^{p^{i-1}}\rangle.
$$
Since $\A^{p^{i-1}}$ commutes with $\C$, replacing $\Omega$ by
$\delta_s\Omega$, where $\delta_s\in\mathrm{Inn}(\S)$ for a suitable $s\in \langle\A\rangle$, we obtain
$S=1$ in (\ref{abce}).

We next go back to (\ref{tig}) with $C=1=D$. Using $Q\in\langle \B^{p^{n-i-1}}\rangle$, we infer
$R^{ep}=1$ and therefore
\begin{equation}
\label{raton2}
R\in\langle\D^{p^{i-1}}\rangle.
\end{equation}
We next replace $\Omega$ by $\delta_s Z^l\Omega$, $\delta_s\in\mathrm{Inn}(\S)$, for suitable $l\in\Z$ and
$s\in \langle \C\rangle$.
This will keep $C=1=D$ as well as $r=1$; moreover, because of (\ref{raton}), will further achieve $P=\A$ and $B=\B$ in (\ref{abce}).
Taking into account (\ref{qato}), (\ref{ato}), and (\ref{raton2}), we see that $W^aX^bY^c\Omega=1$ for suitable $a,b,c$.
This proves (\ref{autoabc}), and completes the proof of the theorem.
\end{proof}

\begin{theorem}\label{main3} Keep the hypotheses and notation of Theorem \ref{main2}, and take $h_0$ odd. Then
$$
|\mathrm{Aut}(\S)|=p^{2n+2}(p-1), [\mathrm{Aut}(\S):G]=p;
$$
\begin{equation}
\label{gcomp2}
\mathrm{Inn}(\S)=\langle L,M,N\rangle\cong (\Z/p^i\Z\times \Z/p^{n-i-1}\Z)\rtimes \Z/p^{n-i-1}\Z;
\end{equation}
\begin{equation}
\label{gcomp3}
G/\langle L,M,N\rangle\cong \Z/p^i\Z\times [ (\Z/p\Z\times \Z/p\Z\times\Z/p\Z)\rtimes \Z/(p-1)\Z],
\end{equation}
with
\begin{equation}
\label{gcomp}
G=\langle L,M,N\rangle\rtimes \langle V,W,X,Y,Z\rangle\text{ if }i>1;
\end{equation}
$$\Phi=\langle L,M,N,V,W,X,Y\rangle$$ is the kernel
of the natural homomorphism $\Gamma:\mathrm{Aut}(\S)\to \mathrm{Aut}(\S/\S^p)$, and hence a normal subgroup of
$\mathrm{Aut}(\S)$;
$$\Pi=\langle L,M,N,U,V,W,X,Y\rangle=\langle M,N,U,V,W,X,Y\rangle=\Phi\langle U\rangle$$ is a normal Sylow $p$-subgroup of
$\mathrm{Aut}(\S)$;
$$
\mathrm{Aut}(\S)=\Pi\rtimes\langle Z\rangle=\Phi\langle U\rangle\rtimes \langle Z\rangle=\langle M,N,U,V,W,X,Y,Z\rangle.
$$
Moreover, the following relations hold:
$$
L^{p^i}=1, M^{p^{n-i-1}}=1,N^{p^{n-i-1}}=1, {}^N L=L^{1+ep}, {}^N M=M^{1+dp}, LM=ML,
$$
$$
U^p=L, V^{p^i}=1,W^{p}=1, X^{p}=1, Y^{p}=1, Z^{p-1}=1,
$$
$$
VW=WV, VY=YV, VZ=ZV, VX=XV,
$$
$$
{}^Z L=L^{g_0}, {}^Z M=h^{g_0}, ZN=NZ, {}^Y L=L, {}^Y M=M, YN=NY,
$$
$$
{}^X L=L, {}^X M=L^{p^{i-1}} M, XN=NX, {}^W L=L, {}^W M=M, WN=NW,
$$
$$
{}^V L=L, {}^V M=M, {}^V N=N, {}^U L=L, {}^U M=MN^{p^{n-i-2}}, {}^U N=L^{-e}N,
$$
$$
WX=XW, {}^YX=X\text{ if }i>1, {}^YX=LX\text{ if }i=1, WY=YW, {}^Z W=W^{h_0^2}, {}^Z X=X^{g_0^2}, {}^Z Y=Y^{h_0},
$$
as well as the following relations, subject to the indicated conditions:
$$
{}^U Y=V^{ep^{i-1}}Y (i>1), {}^U Y=N^{p^{n-3}}V^{e}Y (i=1, n>3), {}^U W=M^{p^{n-i-2}}W (n>2i+1),
$$
$$
{}^U V=V (n>2i+1),{}^U X=X (n>2i+1), {}^U Z=ZN^{p^{n-i-2}(h_0-1)/2}U^{h_0-1} (n>2i+1).
$$
Furthermore, suppose that $n\neq 2i+1$. Then the above are defining relations for $\mathrm{Aut}(\S)$,
as well as for $\Pi$, once all relations involving $Z$ are removed; $\Pi$ cannot be generated by fewer than 7 elements;
$G$ is not a normal subgroup of $\mathrm{Aut}(\S)$.
\end{theorem}

\begin{proof} The verification of the stated relations is a routine calculation that will be omitted.
We do mention, in connection with the relations involving conjugation by $U$, that $U$ satisfies
$$
\A\mapsto\A,\; \B \A^{-ep^{n-i-2}}\C^{-p^{n-i-2}}\mapsto\B,\; \A^e \C\mapsto\C,
$$
so $U^{-1}$ is given by
$$
\A\mapsto\A,\; \B\mapsto\B \A^{-ep^{n-i-2}}\C^{-p^{n-i-2}},\; \C\mapsto\A^e \C.
$$
It should be noted that if $n>2i+1$ then $U^{-1}$ satisfies $\B\mapsto\B\C^{-p^{n-i-2}}$.

We claim that $G$ has index $p$ in $\mathrm{Aut}(\S)$. Indeed, the definition of $U$ implies $U\notin G$.
Thus, our relations $U^p=L$ and $L^{p^i}=1$ imply that $\langle U\rangle$ is $p$-group satisfying $U\notin G$
but $U^p\in G$, which yields $G\cap\langle U\rangle=\langle U^p\rangle$. On the other hand, we have
$\mathrm{Aut}(\S)=G\langle U\rangle$ by Theorem~\ref{main2}, and therefore
$[\mathrm{Aut}(\S):G]=[\langle U \rangle:G\cap\langle U\rangle]=[\langle U \rangle:\langle U^p\rangle]=p$. This
proves the claim. Thus $|\mathrm{Aut}(\S)|=p^{2n+2}(p-1)$ is equivalent to $|G|=p^{2n+1}(p-1)$, which follows
immediately from (\ref{gcomp2}) and (\ref{gcomp3}). Now (\ref{gcomp2}) is a consequence
of Lemma \ref{centro2i}, and we proceed to prove (\ref{gcomp3}). Our relations yield an epimorphism
$$
\Z/p^i\Z\times [ (\Z/p\Z\times \Z/p\Z\times\Z/p\Z)\rtimes \Z/(p-1)\Z]\to G/\langle L,M,N\rangle\cong
\langle \overline{V},\overline{W},\overline{X},\overline{Y},\overline{Z}\rangle,
$$
and we need to see that this is injective. This is equivalent to the following: if $v\in\langle V\rangle$,
$w\in\langle W\rangle$, $x\in\langle X\rangle$, $y\in\langle Y\rangle$, $z\in\langle Z\rangle$, and
$u=vwxyz\in\mathrm{Inn}(\S)$ then $v=w=x=y=z=1$. Suppose then that such $u$ is an inner automorphism of $\S$,
associated to an element $t$ of $\S$. Successively applying $u$ to $\C$, $\B$, and $\A$, we see that
$v=1$, $x=1=z$, and $w=1$. Thus $u=y=Y^l$ for some $l$. Hence $t=\A^i\B^j\C^k$ must commute with $\B$
and $\C$. As $t$ commutes with $\C$, we infer $v_p(j)\geq n-i-1$. But $\B^{p^{n-i-1}}\in Z(\S)$ by Lemma \ref{centro2i},
and we may assume that $t=\A^i\C^k$. Since $t$ commutes with~$\B$, we deduce $v_p(k)\geq n-i-1$.
But $\C^{p^{n-i-1}}=\D\in Z(\S)$,
so we may assume that $t=\A^i$. Thus $u(\A)=\A$ and $u(\A)=\A\C^{l p^{n-2}}$, whence $v_p(l)\geq 1$
and therefore $y=1$, as required. This proves (\ref{gcomp3}). Note that if $i>1$ our relations show that
$$
\langle V,W,X,Y,Z\rangle=\langle V\rangle\times[(\langle W\rangle\times \langle X\rangle\times \langle Y\rangle)\rtimes
\langle Z\rangle],
$$
and the above argument then proves that $\langle V,W,X,Y,Z\rangle\cap \langle L,M,N\rangle$ is trivial, whence (\ref{gcomp}) holds.

By (\ref{gcomp2}) and (\ref{gcomp3}), $\Phi$ has order $p^{2n+1}$, so it is normal in a Sylow $p$-subgroup of
$\mathrm{Aut}(\S)$ (having index $p$ there). But $\Phi$ is also normalized by $Z$, so $\Phi$ is normal in $\mathrm{Aut}(\S)$.
It is in fact the kernel of $\Gamma$, being clearly included there, and using the fact that the index
$[\mathrm{Aut}(\S):\ker\Gamma]$ is divisible by $p$ and $p-1$ (look at $U$ and $Z$). On the other hand,
Sylow's theorem implies that $\Pi$ is normal in $\mathrm{Aut}(\S)$ (the only positive integer that is a factor of $p-1$ and is
congruent to 1 modulo $p$ is 1).

Suppose next that $n\neq 2i+1$. An abstract group generated by elements $L,M,N,U,V,W$,
$X,Y,Z$ satisfying
the stated relations is easily seen to have order $\leq p^{2n+2}(p-1)$. This implies that the stated relations are
defining relations for $\mathrm{Aut}(\S)$. A similar argument applies to $\Pi$.

The given relations allow us to define an epimorphism $\Pi\to B$, where $B$
is a 7-dimensional vector space over $\Z/p\Z$, so $\Pi$ cannot be generated by fewer than 7 elements.

The last stated relation implies $UZU^{-1}\notin G$, as $Z,N\in G$ but $U\notin G$ and
$h_0\not\equiv 1\mod p$.
\end{proof}

\section{Appendix: The automorphism group of $\U$}\label{ahil}

\begin{lemma}\label{cuno} If $s\in\N$, $r\in\Z$, where $r\neq 0$ and $v_p(r)\geq 1$, then
$$
v_p\left(\left( 1+r\right)^{s}-1\right)=v_p(r)+v_p(s).
$$
\end{lemma}

\begin{proof} We have $r=\ell p^a$ and $s=cp^b$, where $a,b,c,\ell\in\Z$, $a,c\geq 1$, $b\geq 0$, $p\nmid c$ and $p\nmid \ell$.
According to Lemma \ref{uno}, there is some $k\in\Z$ such that
$$
(1+r)^s-1=c\ell p^{a+b}+kp^{2a+b}=p^{a+b}(c\ell+kp^{a}).
$$
Since $c\ell$ is relatively prime to $p$, we infer
$$
v_p((1+r)^s-1)=a+b=v_p(r)+v_p(s).
$$
\end{proof}

\begin{lemma}
\label{dos} Let $a,b\geq 1$. Then the order of the element $z=x^ay^b$ of $\U$ is $p^s$, where
$$
s=\max\{n-v_p(a), n-i-v_p(b), 0\}.
$$
In particular, the order of $z$ is $p^n$ if and only if $p\nmid a$.
\end{lemma}

\begin{proof} Given $m\geq 1$, we have $z^m=x^k y^\ell$, where
$$
k=a\frac{(1+p^i)^{bm}-1}{(1+p^i)^b-1}\text{ and }\ell=bm.
$$
By Lemma \ref{cuno}, we have
$$
v_p\left(\frac{(1+p^i)^{bm}-1}{(1+p^i)^b-1}\right)=v_p(m),
$$
so the smallest $m$ such that $x^k=1$ is $m=1$ if $v_p(a)\geq n$ and $m=p^{n-v_p(a)}$ otherwise. On the other hand, the smallest $m$ such that
$y^\ell=1$ is $m=1$ if $v_p(b)\geq n-i$ and $m=p^{n-i-v_p(b)}$  otherwise.
\end{proof}

\begin{lemma}\label{tres} Let $a,b\geq 1$, where $p\nmid a$, and set $z=x^ay^b\in \U$. Then $\langle z\rangle$ is a normal subgroup of $\U$ if and only if $x^{p^i}\in \langle z\rangle$.
Equivalently, $\langle z\rangle$ is normal if and only if $v_p(b)\geq n-2i$.
\end{lemma}

\begin{proof} Since $y$ and $z$ generate $\U$, it follows that $\langle z\rangle$ is normal if and only if
$yzy^{-1}\in \langle z\rangle$, which translates as $x^{a(1+p^i)}y^b\in \langle z\rangle$ or, alternatively,
$x^{ap^i}z\in \langle z\rangle$, that is, $x^{ap^i}\in \langle z\rangle$, where $x^{p^i}$ is a power of $x^{ap^i}$.

Suppose next $\langle z\rangle$ is normal. Then, by above, $x^{p^i}=z^m$ for some $m\geq 1$. Now $z^{m}=x^k y^\ell$,
where $k,\ell$ are as in the proof of Lemma \ref{dos}. Since $v_p(p^i)=i$, it follows that $v_p(m)=i$. As $y^\ell=1$,
we infer that $v_p(b)\geq n-2i$.

Suppose, conversely, that $v_p(b)\geq n-2i$. Set $m=p^i$. Then $z^m=x^k y^\ell$,
where $k,\ell$ are as in the proof of Lemma \ref{dos}, so that $v_p(k)=i$. Since $y^{p^{n-i}}=1$ and $v_p(b)\geq n-2i$, we infer $y^\ell=1$.
All in all, $z^m=x^{cp^i}$, where $p\nmid c$. Since $x^{p^i}\in\langle x^{cp^i}\rangle$, we deduce
$x^{p^i}\in \langle z\rangle$.
\end{proof}

\begin{lemma}\label{cinco} The assignment
\begin{equation}\label{aluno}
y\mapsto y,\; x\mapsto\begin{cases} xy^{p^{n-2i}} & \text{ if } 2i\leq n,\\ xy & \text{ if } 2i\geq n,
\end{cases}
\end{equation}
extends to an automorphism, say $\alpha$, of $\U$ having order $p^i$ if $2i\leq n$ and $p^{n-i}$ if $2i\geq n$.
\end{lemma}

\begin{proof} Set $m=1+p^i$  as well as $b=p^{n-2i}$ if $2i\leq n$, and $b=1$ if $2i\geq n$. Taking into account Lemma \ref{dos},
it suffices to verify that
$$
y(xy^{b})y^{-1}=(xy^b)^m.
$$
On the one hand, we have $y(xy^b)y^{-1}=x^m y^b$, and on the other  $(xy^b)^m=x^k y^\ell$, where
$$
k=1+m^b+m^{2b}+\cdots+m^{p^i b}\text{ and }\ell=bm.
$$
Since $y^{p^{n-i}}=1$, it follows that $y^\ell=y^b$. Set $j=n-i$ if $2i\leq n$ and $j=i$ if $2i\geq n$.
Then Lemma~\ref{uno} yields
$$
k\equiv 1+p^i+p^{j}(1+2+\cdots+p^i)\equiv 1+p^i+p^{j}p^i(p^i+1)/2\mod p^n.
$$
Since $i+j\geq n$ and $p$ is odd, it follows that $k\equiv 1+p^i\mod p^n$, as required.
\end{proof}

Since $\Aut(C_{p^n})$ is abelian, it is clear that $\U$ is a normal subgroup of
the holomorph of $C_{p^n}$, namely the semidirect product
$\H(C_{p^n})=C_{p^n}\rtimes\mathrm{Aut}(C_{p^n})$ of $C_{p^n}$ by its full automorphism group. Thus,
for any integer $r$ relatively prime to $p$ there is an automorphism, say
$\Omega_r$, of $\U$ defined by
\begin{equation}
\label{omegar}
x\mapsto x^r,\; y\mapsto  y,
\end{equation}
namely the restriction to $\U$ of a suitable inner automorphism of $C_{p^n}\rtimes\mathrm{Aut}(C_{p^n})$.

We proceed to select integers $g,h,t,d$ and $e$ for further use within this section.
It is well known that the group of units $(\Z/p^n\Z)^\times$ is cyclic and we fix throughout
\begin{equation}
\label{orderg}
\text{an integer }g\text{ that has order }p^{n-1}(p-1)\text{ modulo }p^n.
\end{equation}
We accordingly set
\begin{equation}\label{aldos}
\beta=\Omega_g,
\end{equation}
and let $h$ be an odd positive integer satisfying
\begin{equation}
\label{eqh}
gh\equiv 1\mod p^n.
\end{equation}
By Lemma \ref{uno}, $1+p^i$ has order $p^{n-i}$ modulo $p^n$. Since $g^{p^{i-1}(p-1)}$ also has order $p^{n-i}$ modulo $p^n$,
there is an integer $t$, relatively prime to $p$, such that
\begin{equation}
\label{eqt}
g^{p^{i-1}(p-1)t}\equiv 1+p^i\mod p^n.
\end{equation}
Thus
\begin{equation}
\label{releq3}
\beta^{p^{i-1}(p-1)t}=\delta_y.
\end{equation}
Now $g^{(p-1)t}$ and $h^{(p-1)t}$ have order $p^{n-1}$ modulo $p^n$, so there exist integers $d$ and $e$ satisfying
\begin{equation}\label{nixp}
g^{(p-1)t}=1+dp, h^{(p-1)t}=1+ ep
\end{equation}
as well as (\ref{eqrd}) and (\ref{eqrd2}).

\begin{prop}\label{gen} We have
$$
\Aut(\U)=\mathrm{Inn}(\U)\langle \alpha, \beta\rangle.
$$
\end{prop}

\begin{proof} Let $\gamma\in \Aut(\U)$. By Lemma \ref{dos}, we have $\gamma(x)=x^a y^b$, where $p\nmid a$.
Thus $\beta^r\gamma(x)=xy^b$ for some $r$. As $\langle x\rangle$ is normal in $\U$,
it follows from Lemmas \ref{tres} and \ref{cinco} that $\alpha^s\beta^r\gamma(x)=x$
for some $s$. Since $\alpha^s\beta^r\gamma$ must preserve the relation $yxy^{-1}=x^{1+p^i}$, we see that
$\alpha^s\beta^r\gamma(y)=x^c y$ for some $c$. Here Lemma \ref{dos} implies that $v_p(c)\geq i$, so there is an integer $t$ such that
$\delta_x^t\alpha^s\beta^r\gamma$ fixes $x$ and $y$. Since $\mathrm{Inn}(\U)$ is normal in $\mathrm{Aut}(\U)$,
we infer that $\gamma\in\mathrm{Inn}(\U)\langle \alpha, \beta\rangle$.
\end{proof}

\begin{lemma}\label{relacion} Let $r$ be any integer relatively prime to $p$, and let $s$ be any odd positive integer
satisfying
$$
rs\equiv 1\mod p^n,
$$
whose existence is guaranteed by the fact that $p^n$ is odd.
Let $\gamma=\Omega_r\in\mathrm{Aut}(\U)$, as defined in~(\ref{omegar}), and set
$$
f=\begin{cases} p^{n-2i}(s-1)/2 & \text{ if }2i\leq n,\\
(s-1)/2  & \text{ if } 2i \geq n.
\end{cases}
$$
Then
$$
\gamma\alpha\gamma^{-1}=\delta_y^f \alpha^s.
$$
\end{lemma}

\begin{proof} Set $$
j=\begin{cases} i & \text{ if }2i\leq n,\\
 n-i & \text{ if } 2i\geq n,
\end{cases}\;
\ell=\begin{cases} p^{n-2i} & \text{ if }2i\leq n,\\
1 & \text{ if } 2i \geq n.
\end{cases}
$$
We then have
$$
\gamma\alpha \gamma^{-1}(x)=\gamma\alpha(x^s)=\gamma(xy^\ell)^s=x^{rk} y^{\ell s},
$$
where
$$
k=1+(1+p^i)^\ell+\cdots+(1+p^i)^{\ell(s-1)}.
$$
Here Lemma \ref{uno} yields
$$
rk\equiv r[s+p^{n-j}(1+2+\cdots+(s-1))]\equiv 1+p^{n-j}r s(s-1)/2\equiv 1+p^{n-j}(s-1)/2\mod p^n.
$$
On the other hand, we have
$$
\delta_y^f \alpha^s(x)=\delta_y^f(xy^{\ell s})=\delta_y^f(x)y^{\ell s}=x^{(1+p^i)^f}y^{\ell s},
$$
where, according to Lemma \ref{uno},
$$
(1+p^i)^f\equiv 1+p^{n-j}(s-1)/2\mod p^n.
$$
\end{proof}

Recalling the meanings of $\beta$ and $h$ from (\ref{aldos}) and (\ref{eqh}), respectively,
Lemma \ref{relacion} ensures that if $2i\leq n$, then
\begin{equation}
\label{releq1}
\beta\alpha\beta^{-1}=\delta_y^{p^{n-2i}(h-1)/2}\alpha^h,
\end{equation}
while if $2i\geq n$ then
\begin{equation}
\label{releq2}
\beta\alpha\beta^{-1}=\delta_y^{(h-1)/2}\alpha^h.
\end{equation}

\begin{lemma}\label{order} The order of $\overline{\beta}\in \mathrm{Out}(\U)$ is $p^{i-1}(p-1)$.
\end{lemma}

\begin{proof} This is the order of $g\langle 1+p^i\rangle\in (\Z/p^{n}\Z)^\times/\langle 1+p^i\rangle$.
\end{proof}

\begin{prop}\label{presout} If $2i\leq n$ then $\mathrm{Out}(\U)$ has order $p^{2i-1}(p-1)$ and presentation
$$
\mathrm{Out}(\U)\cong \langle a,b\,|\, a^{p^i}=1, b^{p^{i-1}(p-1)}=1, bab^{-1}=a^h\rangle\cong \mathrm{Hol}(\Z/p^{i}\Z),
$$
while if $2i\geq n$ then $\mathrm{Out}(\U)$ has order $p^{n-1}(p-1)$, presentation
$$
\mathrm{Out}(\U)\cong \langle a,b\,|\, a^{p^{n-i}}=1, b^{p^{i-1}(p-1)}=1, bab^{-1}=a^h\rangle,
$$
and is isomorphic to
$$
\mathrm{Hol}(\Z/p^{i}\Z)/(p^{n-i}\Z/p^{i}\Z)\cong (\Z/p^{n-i}\Z)\rtimes (\Z/p^{i}\Z)^\times.
$$
\end{prop}

\begin{proof} Let $F$ be the free group on $\{a,b\}$ and consider the homomorphism $\Delta: F\to \mathrm{Out}(\U)$
$$
a\mapsto \overline{\alpha},\; b\mapsto \overline{\beta}.
$$
This is surjective by Proposition \ref{gen}. Set $j=i$ if $2i\leq n$ and $j=n-i$ if $2i\geq n$. By Lemma \ref{cinco},
Lemma \ref{order}, and equations (\ref{releq1})-(\ref{releq2}) we have
$Y=\{a^{p^{j}},b^{p^{i-1}(p-1)},bab^{-1}a^{-h}\}\subset\ker\Delta$.
Let $N$ be the normal closure of $Y$ in $F$ and let $\Gamma:F/N\to \mathrm{Out}(\U)$ be epimorphism associated to $\Delta$.
It is clear that $|F/N|\leq p^{i+j-1}(p-1)$. On the other hand, the definitions (\ref{aluno}), (\ref{omegar}), and
(\ref{aldos}) of $\alpha$ and $\beta$ yield that  $\beta^k=\alpha^\ell\delta_u$ implies $\alpha^\ell=1$. In particular,
$\alpha$ and $\overline{\alpha}$ have the same order, and
$\langle\overline{\alpha}\rangle\cap \langle\overline{\beta}\rangle$ is  trivial.
Thus, Lemmas \ref{cinco} and \ref{order} give
$|\mathrm{Out}(\U)|=p^{i+j-1}(p-1)$, whence $\Gamma$ is an isomorphism.
\end{proof}

\begin{lemma}\label{nueve} We have $Z(\U)=\langle x^{p^{n-i}}\rangle$, so that
$\mathrm{Inn}(\U)$ has order $p^{2(n-i)}$ and
presentation
$$
\mathrm{Inn}(\U)=\langle u,v\,|\, u^{p^{n-i}}=1, v^{p^{n-i}}=1, vuv^{-1}=u^{1+p^i}\rangle\cong C_{p^{n-i}}\rtimes C_{p^{n-i}}.
$$
This is an abelian group if and only if $2i\geq n$.
\end{lemma}

\begin{proof} Since $\langle y\rangle$ acts faithfully on $\langle x\rangle$, the central elements of $\U=\langle x\rangle\rtimes
\langle y\rangle$ are the elements of $\langle x\rangle$ that are fixed by all elements of $\langle y\rangle$, which are readily
seen to be the powers of $x^{p^{n-i}}$.
\end{proof}

\begin{lemma}\label{gcom} Let $G=\langle a,b,c\rangle$ be a group, where $c$ normalizes $\langle a,b\rangle$,
$a$ and $b$ commute with $[a,b]$, and $[a,b]\in \langle c\rangle$. Then $G=\langle a\rangle\langle b\rangle \langle c\rangle$.
\end{lemma}

\begin{proof} By hypothesis, $G=\langle a,b\rangle\langle c\rangle$ and
$\langle a,b\rangle=\langle a\rangle\langle b\rangle\langle [a,b]\rangle$.  As $[a,b]\in \langle c\rangle$, the result follows.
\end{proof}

\begin{theorem}\label{ocho} The order of $\mathrm{Aut}(\U)$ is $p^{2n-1}(p-1)$ if $2i\leq n$,
and $p^{3n-2i-1}(p-1)$ if $2i\geq n$. Moreover, if
\begin{equation}
\label{eqmu}
\mu=\delta_x^{\ell}\beta\delta_x^{-\ell},
\end{equation}
then $\mathrm{Aut}(\U)$ is generated by $\alpha,\delta_x,\mu$ for any $\ell\in\Z$, and the choice
\begin{equation}
\label{eqell}
2\ell\equiv 1\mod p^n,
\end{equation}
yields the following defining relations, depending on whether $2i\leq n$ or $2i\geq n$:
\begin{equation}
\label{defrel}
\alpha^{p^i}=1, \delta_x^{p^{n-i}}=1,\mu^{p^{n-1}(p-1)}=1,
\alpha\delta_x\alpha^{-1}= \delta_x \mu^{p^{n-i-1}(p-1)t}, \mu \delta_x \mu^{-1}=\delta_x^g,
\mu\alpha \mu^{-1}=\alpha^h,
\end{equation}
\begin{equation}
\label{defrelmayor}
\alpha^{p^{n-i}}=1, \delta_x^{p^{n-i}}=1, \mu^{p^{n-1}(p-1)}=1,
\alpha\delta_x\alpha^{-1}= \delta_x \mu^{p^{i-1}(p-1)t}, \mu \delta_x \mu^{-1}=\delta_x^g,
\mu\alpha \mu^{-1}=\alpha^h.
\end{equation}
Moreover, if $\sigma\in\mathrm{Aut}(\U)$, then $\sigma=ABC$ for unique elements $A\in\langle\alpha\rangle$,
$B\in\langle\delta_x\rangle$, and $C\in\langle\mu\rangle$.
\end{theorem}

\begin{proof} Set
$$
j=\begin{cases} i & \text{ if }2i\leq n,\\
n-i & \text{ if }2i\geq n,\quad
\end{cases}
k=\begin{cases} p^{n-2i} & \text{ if }2i\leq n,\\
1 & \text{ if }2i\geq n.
\end{cases}
$$
By Proposition \ref{presout} and Lemma \ref{nueve}, we have
$$
|\mathrm{Out}(\U)|=p^{i-1+j}(p-1),\; |\mathrm{Inn}(\U)|=p^{2(n-i)},
$$
so
$$
|\mathrm{Aut}(\U)|=p^{2(n-i)+i-1+j}(p-1),
$$
as claimed.

We have $\delta_y=\beta^{p^{i-1}(p-1)t}$ by (\ref{releq3}), so $\alpha,\delta_x,\beta$
generate $\mathrm{Aut}(\U)$ by Proposition \ref{gen}. We next verify the analogs of (\ref{defrel})-(\ref{defrelmayor})
for these generators. By Lemma \ref{cinco}, we have $\alpha^{p^j}=1$. Since $g$ has order $p^{n-1}(p-1)$ modulo $p^n$, it follows that
$\beta^{p^{n-1}(p-1)}=1$. On the other hand, Lemma \ref{nueve} implies $\delta_x^{p^{n-i}}=1$. Moreover, according to
Lemma \ref{cinco},
\begin{equation}
\label{conh}
\alpha\delta_x\alpha^{-1}=\delta_{\alpha(x)}=\delta_x\delta_y^{k}=\delta_x\beta^{k p^{i-1}(p-1)t},
\end{equation}
while the very definition of $\beta$ yields
$$\beta \delta_x \beta^{-1}=\delta_{\beta(x)}=\delta_x^g.$$
Furthermore, by (\ref{releq1}) and
(\ref{releq2}), we have
\begin{equation}
\label{ab}
\beta\alpha\beta^{-1}=\delta_y^{k(h-1)/2}\alpha^h.
\end{equation}
This completes the required verification.

It is clear that for any integer $\ell$, the elements $\alpha,\mu,\delta_x$ still generate $\mathrm{Aut}(\U)$.
We next show that (\ref{defrel})-(\ref{defrelmayor}) hold for a suitable choice of $\ell$.
It is evident that
$$
\mu^{p^{n-1}(p-1)}=1,\; \mu \delta_x \mu^{-1}=\delta_x^g,
$$
Observe next that
\begin{equation}
\label{jata}
\delta_y^{k}\in Z(\mathrm{Aut}(\U)).
\end{equation}
Indeed, we have
$$
\delta_y^{k}\delta_x\delta_y^{-k}=\delta_{u},
$$
where, by Lemma \ref{uno},
$$
u=x^{(1+p^i)^{k}}=x^{1+p^{n-j}}.
$$
Since $x^{p^{n-j}}\in Z(\U)$, it follows that $\delta_y^{k}$ commutes with $\delta_x$. As $\delta_y$
is a power of $\beta$, they commute. Since $\alpha(y)=y$, it follows that $\alpha$ and $\delta_y$ also commute.
This proves (\ref{jata}).

Since $\delta_x^\ell$ commutes with  $\delta_y^{k}=\beta^{k p^{i-1}(p-1)t}$, it follows
that
\begin{equation}
\label{upardo}
\mu^{k p^{i-1}(p-1)t}=\beta^{k p^{i-1}(p-1)t},
\end{equation}
whence (\ref{conh}) gives
\begin{equation}
\label{upa}
\alpha\delta_x\alpha^{-1}= \delta_x \mu^{k p^{i-1}(p-1)t}.
\end{equation}

Suppose next that $2\ell\equiv 1\mod p^n$. We claim that
$$
\mu\alpha \mu^{-1}=\alpha^h.
$$
Indeed, conjugating each side of (\ref{ab}) by $\delta_x^\ell$ and using that $\delta_y^{k}$ is central yields
\begin{equation}
\label{urra}
\mu(\delta_x^\ell\alpha\delta_x^{-\ell})\mu^{-1}=\delta_y^{k (h-1)/2}(\delta_x^\ell\alpha\delta_x^{-\ell})^h.
\end{equation}
To compute with (\ref{urra}), note that $[\delta_x^{-1},\alpha]=\delta_y^k$ by (\ref{conh}). As $\delta_y^{k}$ is central,
$[\delta_x,\alpha]^{-1}=\delta_y^k$, so
$$
\delta_x\alpha\delta_x^{-1}=\delta_y^{-k}\alpha,
$$
and therefore
\begin{equation}
\label{jara}
\delta_x^r\alpha\delta_x^{-r}=\delta_y^{-r k}\alpha,\quad r\in\Z.
\end{equation}
Successively using (\ref{jata}), (\ref{jara}), (\ref{urra}), (\ref{jara}), and (\ref{jata}), we infer that
$$
\delta_y^{-\ell k}\mu\alpha\mu^{-1} =\mu\delta_y^{-\ell k}\alpha\mu^{-1}=\mu(\delta_x^\ell\alpha\delta_x^{-\ell})\mu^{-1}=
\delta_y^{k (h-1)/2}(\delta_x^\ell\alpha\delta_x^{-\ell})^h=\delta_y^{k (h-1)/2}\delta_y^{-\ell h k}\alpha^h.
$$
Since $\ell(h-1)\equiv (h-1)/2\mod p^n$, the claim follows.

Thus, the generators $\alpha,\mu,\delta_x$ satisfy all of the relations (\ref{defrel}) if $2i\leq n$ and (\ref{defrelmayor}) if
$2i\geq n$.

Let $G$ any group generated by elements $\alpha,\mu,\delta_x$ that satisfy (\ref{defrel}) if $2i\leq n$ and (\ref{defrelmayor}) if $2i\geq n$. Set $\delta_y=\mu^{p^{i-1}(p-1)t}$. Then (\ref{defrel}) and (\ref{defrelmayor}) imply that
$\langle \delta_x,\delta_y\rangle$
is a normal subgroup of $G$ of order $\leq p^{2(n-i)}$, with
$|G/\langle \delta_x,\delta_y\rangle|\leq p^{i-1+j}(p-1)$, so that $|G|\leq p^{2(n-i)+i-1+j}(p-1)$. This shows that (\ref{defrel}) and
(\ref{defrelmayor})
are defining relations for $\mathrm{Aut}(\U)$.

It follows from Lemma \ref{gcom} that $\mathrm{Aut}(\U)=\langle \alpha\rangle\langle \beta\rangle\langle \mu\rangle$.
As $|\mathrm{Aut}(\U)|=|\langle \alpha\rangle||\langle \beta\rangle||\langle \mu\rangle|$, we obtain the stated
normal form for the elements of $\mathrm{Aut}(\U)$.
\end{proof}

Let $\mu$ be defined as in (\ref{eqmu}) with $\ell$ chosen so that (\ref{eqell}) holds. Recalling the meaning $t$ from
(\ref{eqt}), we set
\begin{equation}\label{eqnul}
\nu=\mu^{(p-1)t}.
\end{equation}


\begin{prop}\label{Syl} Let $\So$ be the subgroup of $\mathrm{Aut}(\U)$ generated by $\alpha,\delta_x,\nu$.
Then $\So$ is a normal Sylow $p$-subgroup of $\mathrm{Aut}(\U)$ with the following defining relations,
depending on whether $2i\leq n$ or $2i\geq n$:
\begin{equation}
\label{defrel2}
\alpha^{p^i}=1, \nu^{p^{n-1}}=1, \delta_x^{p^{n-i}}=1,
\alpha\delta_x\alpha^{-1}= \delta_x \nu^{p^{n-i-1}}, \nu \delta_x \nu^{-1}=\delta_x^{1+d p},
\nu\alpha \nu^{-1}=\alpha^{1+e p},
\end{equation}
\begin{equation}
\label{defrelmayor2}
\alpha^{p^{n-i}}=1, \nu^{p^{n-1}}=1, \delta_x^{p^{n-i}}=1,
\alpha\delta_x\alpha^{-1}= \delta_x \nu^{p^{i-1}}, \nu \delta_x \nu^{-1}=\delta_x^{1+d p},
\nu\alpha \nu^{-1}=\alpha^{1+e p}.
\end{equation}
In particular, $\So$ is isomorphic to $\S$.
\end{prop}

\begin{proof} Theorem \ref{ocho} implies that $\So$ is a normal subgroup of $\mathrm{Aut}(\U)$
and that
$$
\mathrm{Aut}(\U)/\So\cong \So\langle\mu\rangle/\So\cong
\langle\mu\rangle/(\So\cap\langle\mu\rangle)\cong \langle\mu\rangle/\langle\nu\rangle\cong C_{p-1},
$$
so $\So$ is a Sylow $p$-subgroup of $\mathrm{Aut}(\U)$.

The relations (\ref{defrel2})-(\ref{defrelmayor2})
follow easily from the relations (\ref{defrel})-(\ref{defrelmayor}). Set $j=i$ if $2i\leq n$ and $j=n-i$ if $2i\geq n$.
That (\ref{defrel2})-(\ref{defrelmayor2}) are defining relations follows from
the fact that any group $G$  generated by elements $\alpha,\nu,\delta_x$ satisfying these relations has order $\leq p^{2(n-i)+i-1+j}$,
since $\langle \delta_x,\nu^{p^{i-1}}\rangle$ is normal in $G$ of order $\leq p^{2(n-i)}$ with
$|G/\langle \delta_x,\nu^{p^{i-1}}\rangle|\leq p^{i-1+j}$.
\end{proof}



\noindent{\bf Acknowledgment.} We are very indebted to the referee for a careful reading of the paper and
valuable feedback.


\end{document}